\documentclass[11pt,a4paper,twoside,reqno]{amsart}
\usepackage[utf8]{inputenc}
\usepackage[T1]{fontenc}
\usepackage{amsfonts,amsmath,amssymb,amsthm} 
\usepackage{mathtools}
\usepackage[paper=a4,DIV=11,BCOR=15mm]{typearea}
\usepackage[linktocpage, pdftex,colorlinks]{hyperref}
    \hypersetup{colorlinks,
    linkcolor=blue,
    filecolor=blue,
    urlcolor=blue,
    citecolor=blue,
    pdftitle={Observability for Non-autonomous Systems},
    pdfauthor={Clemens Bombach, Fabian Gabel, Christian Seifert, Martin Tautenhahn},
    pdfkeywords={Banach space, evolution family, non-autonomous system, null-controllability, observability, strongly elliptic, Ornstein-Uhlenbeck operator},
    pdfsubject={MSC2020 Primary: 35Q93, 47N70; Secondary: 93B07, 93B28},
    }
\usepackage{xcolor} 
\usepackage{changebar}
\colorlet{ins}{blue} 
\colorlet{del}{red}

\usepackage{enumerate}
\usepackage[textsize=tiny]{todonotes}

\newcommand{\re}{{\ensuremath{\mathrm{Re\,}}}}
\renewcommand{\Re}{\re}

\DeclareMathOperator{\trace}{\mathrm{tr}}

\newcommand{\euler}{\mathrm{e}}
\newcommand{\spann}{\mathrm{span}}

\newcommand{\fa}{\mathfrak{a}}
\newcommand{\fb}{\mathfrak{b}}

\newcommand{\cF}{\mathcal{F}}
\newcommand{\cO}{\mathcal{O}}

\newcommand{\cL}{\mathcal{L}}
\newcommand{\cS}{\mathcal{S}}

\newcommand{\ee}{\mathrm e}
\newcommand{\ii}{\mathrm i}
\newcommand{\id}{\mathrm{Id}}

\newcommand{\eps}{\varepsilon}

\newcommand{\abs}[1]{\left| #1 \right|} 
\newcommand{\scalarproduct}[2]{\langle #1, #2 \rangle}
\newcommand{\norm}[1]{\|#1\|}

\newcommand{\ska}[1]{\langle #1 \rangle}

\newcommand{\N}{\mathbb{N}}
\newcommand{\R}{\mathbb{R}}
\newcommand{\C}{\mathbb{C}}
\newcommand{\Z}{\mathbb{Z}}

\newcommand{\1}{\mathbf{1}}

\DeclareMathOperator*{\esssup}{ess\,sup}
\newcommand{\obs}{\mathrm{obs}}
\newcommand{\Ell}{\mathrm{L}}
\newcommand{\WW}{\mathrm{W}}
\newcommand{\BB}{\mathrm{B}}
\newcommand{\CC}{\mathrm{C}}

\renewcommand{\d}[1]{\ensuremath\, {\operatorname{d}\!{#1}}}

\newcommand{\one}{\mathbf{1}}

\newcommand{\trans}{T}

\newcommand{\thickset}{\Omega}
\newcommand{\from}{\colon}

\DeclareMathOperator{\supp}{supp}

{\bf}{\it}
{\bf}{\it}

\newtheorem{thm}{Theorem}[section]

\newtheorem{lem}[thm]{Lemma}

\newtheorem{prop}[thm]{Proposition}
\newtheorem{cor}[thm]{Corollary}
\newtheorem{hypo}[thm]{Hypothesis}

\theoremstyle{definition}
\newtheorem{defn}[thm]{Definition}
\newtheorem{example}[thm]{Example}

\theoremstyle{remark}
\newtheorem{rem}[thm]{Remark}
\usepackage{changes}
\usepackage{tcolorbox}

\begin{document}

\title[Observability for Non-autonomous Systems]{Observability for Non-autonomous Systems}
\author[C.~Bombach]{Clemens Bombach}
\address[Clemens Bombach]{
Technische Universit\"at Chemnitz,
Fakult\"at f\"ur Mathematik,
Reichenhainer Straße 41, 
 D-09126 Chemnitz,
 Germany
 }
\email{\href{mailto:clemens.bombach@mathematik.tu-chemnitz.de}{clemens.bombach@mathematik.tu-chemnitz.de}}
\author[F.~Gabel]{Fabian Gabel}
\address[Fabian Gabel]{
Technische Universit\"at Hamburg,
Institut f\"ur Mathematik,
Am Schwar\-zen\-berg-Campus 3,
D-21073 Hamburg,
Germany}
\email{\href{mailto:fabian.gabel@tuhh.de}{fabian.gabel@tuhh.de}}
\author[C.~Seifert]{Christian Seifert}
\address[Christian Seifert]{
Christian-Albrechts-Universit\"at zu Kiel,
Mathematisches Seminar,
Heinrich-Hecht-Platz 6,
D-24118 Kiel,
Germany,
and
Technische Universit\"at Hamburg,
Institut f\"ur Mathematik,
Am Schwarzenberg-Campus 3,
D-21073 Hamburg,
Germany}
\email{\href{mailto:christian.seifert@tuhh.de}{christian.seifert@tuhh.de}}
\author[M.~Tautenhahn]{Martin Tautenhahn}
\address[Martin Tautenhahn]{
Universit\"at Leipzig,
Mathematisches Institut,
Augustusplatz 10,
D-04109 Leipzig,
Germany
}
\email{\href{mailto:martin.tautenhahn@math.uni-leipzig.de}{martin.tautenhahn@math.uni-leipzig.de}}

\begin{abstract}
We study non-autonomous observation systems
\begin{align*}
  \dot{x}(t) = A(t) x(t),\quad y(t) = C(t) x(t),\quad x(0) = x_0\in X,
\end{align*}
where $(A(t))$ is a strongly measurable family of closed operators on a Banach space $X$ and $(C(t))$ is a family of bounded observation operators from $X$ to a Banach space $Y$.
Based on an abstract uncertainty principle and a dissipation estimate, we prove that the observation system satisfies a final-state observability estimate in $\mathrm{L}^r(E; Y)$ for measurable subsets $E \subseteq [0,T], T > 0$.
We present applications of the above result to families $(A(t))$ of uniformly strongly elliptic differential operators as well as non-autonomous Ornstein--Uhlenbeck operators $P(t)$ on $\mathrm{L}^p(\mathbb{R}^d)$ with observation operators $C(t)u = u|_{\Omega(t)}$.
In the setting of non-autonomous strongly elliptic operators, we derive necessary and sufficient geometric conditions on the family of sets $(\Omega(t))$ such that the corresponding observation system satisfies a final-state observability estimate.
\end{abstract}
\maketitle
\section{Introduction}\label{sec:introduction}
Let $X$ and $Y$ be Banach spaces, $T>0$, $(A(t))_{t\in[0,T]}$ a family of operators $A(t) \colon \allowbreak  D(A(t)) \to X$ on $X$, and $(C(t))_{t\in[0,T]}$ a family of bounded operators $C(t) \colon X \to Y$.
In this article, we will be concerned with systems of the form
\begin{align}\label{eq:observationSystem}
\begin{split}
  \dot{x}(t) &= -A(t) x(t), \quad t \in (0,T], \quad x(0) = x_0 ,\\
    y(t) &= C(t) x(t), \quad t \in [0,T],
\end{split}
\end{align}
where the first equation in~\eqref{eq:observationSystem} describes the evolution of a state function $x$ which is driven by the operators $A(t)$ and the second equation describes the observation function $y$ of the state function $x$ through the operators $C(t)$.
In particular, we are interested in the following question: 
for a given measurable subset $E \subseteq [0,T]$ and $r \in [1,\infty]$, does there exist a constant $C_\obs \geq 0$ such that, for all initial values $x_0 \in X$, the observability estimate  
\begin{align}\label{eq:obs_intro}
\norm{x(T)}_X \leq C_{\obs}
\begin{cases}
   \Bigl(\int_{E} \norm{y(t)}_{Y}^r \d t\Bigr)^{1/r}, & r\in [1,\infty),\\
    \esssup_{t\in E} \norm{y(t)}_{Y}, & r=\infty ,
\end{cases}
\end{align}
holds?
In this case, we say that the system~\eqref{eq:observationSystem} \emph{satisfies a final-state observability estimate in $\Ell^r(E;  Y)$}.
Loosely speaking, final-state observability allows one to retrieve information about the final-state $x(T)$ by just observing the system through the measurements $y(t)$ at times~$t \in E$.

In case the families $(A(t))_{t\in[0,T]}$ and $(C(t))_{t\in[0,T]}$ are constant, final-state observability for \eqref{eq:observationSystem} has been studied thoroughly in the Hilbert space case both in abstract and in concrete situations. Autonomous self-adjoint Schr\"odinger operators $A$ in $\Ell^2$ on bounded domains of $\R^d$ and a projection $Cu  = u|_{\Omega}$ for suitable subsets $\Omega$ of the domain were considered in \cite{LebeauR-95, FursikovI-96}, as well as in \cite{Miller-04,Miller-04b,GonzalezT-07,Barbu-14} for the case of unbounded domains. 
Moreover, second order elliptic operators $A$ in $\Ell^2$ for bounded domains of $\R^d$ have been studied in \cite{PhungW-13}, where also $E$ in~\eqref{eq:obs_intro} is just a measurable subset of $[0,T]$. 
Of particular interest in understanding the case of unbounded domains is the specification of necessary and sufficient geometric conditions on $\Omega$ for observability, which were established in 
\cite{EgidiV-18,WangWZZ-19} in the case of Hilbert spaces. 
In the Banach space case, a characterization of observability in terms of geometric conditions was given in \cite{GallaunST-20,BombachGST-20}.

\subsection{Abstract Observability and Applications}
The main theorem of this article is an abstract observability theorem that combines and generalizes the results from~\cite{GallaunST-20,BombachGST-20} and from~\cite{PhungW-13,WangZ-17}.
The main idea is the extension of the \emph{Lebeau--Robbiano strategy} for deriving observability to the setting of non-autonomous problems~\eqref{eq:observationSystem} and evolution families. Excluding technical details, we summarize the philosophical pillars of this approach as follows:

\medbreak
\begin{center}
\begin{tabular}{p{0.9\textwidth}}
\emph{Given an evolution family generated by the non-autonomous operators $A(t)$ and an abstract dissipation and uncertainty estimate for the observation operators $C(t)$, there exists a constant $C_{\mathrm{obs}}$ such that the final-state observability estimate~\eqref{eq:obs_intro} holds.} 
\end{tabular}
\end{center}
\medbreak

All of the above will be made precise in Hypothesis~\ref{hypo:evoFam} and Theorem~\ref{thm:observability}.
In particular, we will show that our theorem allows us to obtain observability estimates for non-autonomous systems in Banach spaces and to include observation on a measurable set $E\subseteq[0,T]$ of time.
Its proof rests on an interpolation estimate for evolution families presented in Theorem~\ref{thm:interpolation}.

We give two applications for this abstract result: non-autonomous elliptic equations and non-autonomous Ornstein--Uhlenbeck equations. In Section~\ref{sec:elliptic}, for the first application, we prove the following result, which is also prototypical for our second application: 
\medbreak
\begin{center}
\begin{tabular}{p{0.9\textwidth}}
\emph{Consider an observation system~\eqref{eq:observationSystem} consisting of a parabolic equation in $\R^d$
\begin{equation*}
 \dot u = -A(t) u =  -\sum_{|\alpha| \leq m} a_\alpha(t) \partial^\alpha u 
\end{equation*}
with time-dependent uniformly elliptic differential operators $A(t)$ 
and observation operators $C(t)$ that are given via the restriction $u|_{\Omega(t)}$ of functions $u$ on $\R^d$ to a time-dependent family of observability sets $\Omega(t) \subseteq \R^d$. 
Then, under suitable geometric assumptions on the family $(\Omega(t))_{t\in[0,T]}$, a final-state observability estimate~\eqref{eq:obs_intro} holds.
}
\end{tabular}
\end{center}
\medbreak
More precisely, we analyze the connection between \emph{final-state observability estimates} and the \emph{geometry} of the observability sets.
First and as an extension of the results known for the autonomous setting, in Theorem~\ref{thm:obs}, we show that a \emph{uniformly thick} family of observability sets $\Omega(t)$ guarantees the existence of a final-state observability estimate.
Then, in Theorem~\ref{thm:obsImpliesMeanThick}, we derive a converse implication to Theorem~\ref{thm:obs} which builds upon a weaker notion of \emph{thickness}.
\par
As our second application, in Section~\ref{sec:OU}, we study non-autonomous Ornstein--Uhlenbeck equations. These equations are well studied in the Hilbert space setting, see, for example, the work \cite{BeauchardEP-S-20} and the references contained therein. They are parabolic equations associated with a family of second-order differential operators taking the form
\begin{equation*}
	P(t) = \frac{1}{2}\trace(A(t)A(t)^\trans\nabla_x^2) - \scalarproduct{B(t)x}{\nabla_x} - \frac{1}{2}\trace(B(t))\,,
\end{equation*}
where $A,B \in \CC^\infty(0,T;\R^{d \times d})$. Note that, in contrast to the elliptic operators considered in Section~\ref{sec:elliptic}, we have some dependence on the space variable in the first order term. In general, these operators are not elliptic. Therefore, we need to assume a certain Kalman rank condition, which is equivalent to hypoellipticity in the autonomous case, i.e., when the matrices $A,B$ are independent of $t$. The main result of this section is Theorem~\ref{thm:OU-obs} which states that for uniformly thick observability sets $\Omega(t)$ and small final times $T$, a final-state observability estimate in $\Ell^p$ holds for the evolution family associated with these Ornstein--Uhlenbeck operators if we assume a Kalman rank condition. In the case of small times, this generalizes an $\Ell^2$-observability estimate from \cite{BeauchardEP-S-20}.

\subsection{Null-controllability}

The above question on final-state observability is closely related to a variant of null-controllability properties of the predual system:
Given a final time $T> 0$, we consider the system
\begin{align}\label{eq:controlSystem}
  \dot{x}(t) &= -A(t) x(t) + B(t) u(t), \quad t \in (0,T], \quad x(0) = x_0,
\end{align}
where $(B(t))_{t\in[0,T]}$ is a family of bounded linear operators $B(t)\colon U\to X$ for some Banach space $U$. Problems like~\eqref{eq:controlSystem} for example appear in the field of controllability of partial differential equations, where $(A(t))_{t\in[0,T]}$ is a family of differential operators and $B(t)=\one_{\thickset(t)}$ is a multiplication operator on $\Ell^p(\R^d)$ which represents moving control subsets $(\thickset(t))_{t\in[0,T]}$, cf.\@ \cite{MartinRR-13,ChavesRZ-14,LeRousseauLTT-17,BeauchardEP-S-20}.

Let us introduce two properties which may hold or not hold.
\begin{enumerate}[(a)]
 \item\label{item:null-unif} For all measurable subsets $E \subseteq [0,T]$, there exists $C > 0$ such that, for all $\varepsilon > 0$ and all $x_0 \in X$, there exists $u \in \Ell^{r}(0,T; U)$ with $\supp u \subseteq E$, $\lVert u \rVert_{\Ell^{r} (0,T;U)} \leq C \lVert x_0 \rVert_X$, and $\lVert x(T) \rVert_X < \varepsilon$. 
 \item\label{item:null-non-unif} For all measurable subsets $E \subseteq [0,T]$ and all $\varepsilon > 0$, there exists $C_\varepsilon > 0$, such that, for all $x_0 \in X$, there exists $u \in \Ell^{r}(0,T; U)$ with $\supp u \subseteq E$, $\lVert u \rVert_{\Ell^{r} (0,T;U)} \leq C_\varepsilon \lVert x_0 \rVert_X$, and $\lVert x(T) \rVert_X < \varepsilon$. 
\end{enumerate}
Note that the constant $C$ in (\ref{item:null-unif}) is not allowed to depend on $\varepsilon$, while the constant $C_\varepsilon$ in (\ref{item:null-non-unif}) may well depend on $\varepsilon$ and even tend to infinity as $\varepsilon$ tends to zero. Obviously, (\ref{item:null-unif}) $\Rightarrow$ (\ref{item:null-non-unif}). Both properties are variants of \emph{approximate null-controllability}. 
As a consequence of Douglas' lemma~\cite{Douglas-66} for Hilbert spaces and corresponding results for Banach spaces \cite{DoleckiR-77,Carja-88,Vieru-05,YuLC-06}, the system~\eqref{eq:controlSystem} satisfies property (\ref{item:null-unif}) if and only if 
the dual system, which is of the form~\eqref{eq:observationSystem} with $A(t)$ replaced by the dual operators $A(t)'$ and $C(t) = B(t)'$ for all $t\in [0,T]$, satisfies a final-state observability estimate in $\Ell^{r/(r-1)}(E; U')$. 
Note that, in certain situations, 
property (\ref{item:null-unif}) is equivalent to null-controllability, that is, there exists $u \in \Ell^{r}(0,T; U)$ with $\supp u \subseteq E$ such that $x(T) = 0$, cf.\ \cite[Remark~2.1]{Carja-88}. Thus (\ref{item:null-unif}) is a natural generalization
of null-controllability from Hilbert spaces to Banach spaces.
Likewise, property (\ref{item:null-non-unif}) is equivalent to a so-called \emph{weak observability estimate} for the dual system. This is spelled out in \cite{Miller-10,TrelatWX-20,AlphonsoM-22} in the setting of Hilbert spaces, and in \cite{EgidiGST-21} for its generalization to Banach spaces.

\subsection{Outline}
Let us outline the content of the article.
After an introduction to the framework of evolution families for non-autonomous Cauchy problems in Section \ref{sec:non-aut}, we will first derive sufficient conditions for observability of abstract non-autonomous systems in Section \ref{sec:obs}.
Building on this abstract result, we will consider the concrete application of non-autonomous elliptic differential operators and relate geometric conditions of the sets of observations to final-state observability in Section \ref{sec:elliptic}. 
Furthermore, we prove an $\Ell^p$-observability estimate for evolution families associated with non-autonomous Ornstein--Uhlenbeck equations in Section \ref{sec:OU}.
We collect some properties of the non-autonomous elliptic differential operators as well as the corresponding evolution families in the appendix.

\section{Linear Non-autonomous Cauchy~Problems and Evolution~Families}
\label{sec:non-aut}

For Banach spaces $X$ and $Y$, let $\cL(X,Y)$ denote the set of \emph{bounded linear operators} from $X$ to $Y$.
Similarly, set $\cL(X) \coloneqq \cL(X,X)$.
Let $T > 0$, and let $A(t) \colon D(A(t)) \to X$, $0 \leq t \leq T$, be a family of operators in $X$.
Furthermore, let us always assume that the operators $A(t)$ are closed with common domain $D(A(t)) = D$ for all $t\in[0,T]$, where $D$ densely embeds into~$X$, and that the mapping $A \colon [0,T] \to \cL(D, X)$ is strongly measurable.

Consider the homogeneous initial value problem
\begin{align}\label{eq:evoProblem}
  \tag{NACP}
    \dot{x}(t) = -A(t) x(t), \quad t \in (0,T],  \quad x(0) = x_0 ,
\end{align}
where $x_0\in X$.
We will call~\eqref{eq:evoProblem} the \emph{non-autonomous Cauchy problem for~$A$}.
Throughout this article, we make use of the following concept of solutions for~\eqref{eq:evoProblem}; see~\cite[Chapter~4, Definition~2.8]{Pazy-83}.
\begin{defn}[Strong Solution]
\label{defn:strongSolution}
  A function $x\from[0,T]\to X$ is said to be a \emph{strong solution of \eqref{eq:evoProblem}} if $x \in \WW^{1,1}(0,T; X) \cap \Ell^1(0,T; D)$, $x(0) = x_0$, and $\dot{x}(t) = -A(t)x(t)$ for almost all $t \in (0,T)$.
\end{defn}

The following definition provides a natural generalization of operator semigroups to the context of non-autonomous Cauchy problems.

\begin{defn}[Evolution Family]\label{defn:evoFam}
    Let $T > 0$. A two-parameter family of bounded linear operators $(U(t,s))_{0 \leq s \leq t \leq T}$ on $X$ is called an \emph{evolution family} if
    \begin{enumerate}[(a)]
      \item\label{it:algebraic} $U(s,s) = \id$ and $U(t,s)U(s,r) = U(t,r)$ for $0\leq r\leq s \leq t \leq T$.
    \end{enumerate}
	If, furthermore,
    \begin{enumerate}[(a)]\addtocounter{enumi}{1}
      \item\label{it:stronglyContinuous} $(t,s) \mapsto U(t,s)$ is strongly continuous for $0 \leq s \leq t \leq T$,
    \end{enumerate}
	then we say that the evolution family~$(U(t,s))$ is \emph{strongly continuous}.

        An evolution family is called \emph{exponentially bounded} if there exist $M\geq 1$ and $\omega \in \R$ such that $\|U(t,s)\|_{\cL(X)} \leq M \ee^{\omega (t - s)}$ for all $0 \leq s \leq t \leq T$.

        An evolution family $(U(t,s))_{0 \leq s \leq t \leq T}$ is called a \emph{(strongly continuous) evolution family for $A$} if, in addition to conditions~\eqref{it:algebraic} (and \eqref{it:stronglyContinuous}), the following conditions are satisfied:
  \begin{enumerate}[(a)]
	\addtocounter{enumi}{2}
      \item\label{it:strongSolution} For all $0\leq s < T$ and $x_s \in D$, the function $x \colon [s,T] \to X$ defined by $x(t) = U(t,s) x_s$ is in $\WW^{1,1}(s,T; X)\cap \Ell^1(s,T; D)$ and satisfies $\dot{x}(t) = -A(t) U(t,s)x_s$ for almost all $t \in (s,T)$.
      \item\label{it:strongSolutionWRTs} For all $0  < t \leq T$ and $x_T \in D$, the function $x \colon [t,T] \to X$ defined by $x(s) = U(T,s) x_T$ is in $\WW^{1,1}(t,T; X)\cap \Ell^1(t,T; D)$ and satisfies $\dot{x}(s) = U(T,s) A(s) x_T$ for almost all $s \in (t,T)$.
  \end{enumerate}
\end{defn}

\begin{rem}
\label{rem:EvolutionFamilies}
  \begin{enumerate}[(a)]
    \item Since the early works by Sobolevski\u{\i}~\cite{Sobolevskii-61} and Tanabe~\cite{Tanabe-60}, non-autonomous Cauchy problems~\eqref{eq:evoProblem} and evolution families have been extensively studied by various authors.
      For details, we refer to~\cite{Tanabe-79,Pazy-83,AcquistapaceT-87,Yagi-91,Lunardi-95,Nickel-97,EngelN-00} and the references therein.
      Most of the aforementioned resources share the algebraic condition~\eqref{it:algebraic} in their definition of the evolution family $(U(t,s))$ but rely on other regularity assumptions~\eqref{it:stronglyContinuous}, \eqref{it:strongSolution}, and \eqref{it:strongSolutionWRTs} and also assume other regularity properties of the operator family $(A(t))$.
      In the literature, an evolution family is also referred to as an \emph{evolution system}, \emph{evolution operator}, \emph{evolution process}, \emph{propagator}, or \emph{fundamental solution}.
    \item\label{rem:EvolutionFamilies:item:3} The condition in Definition~\ref{defn:evoFam}~\eqref{it:strongSolution} states that $u(t) \coloneqq U(t,0)u_0$ defines a strong solution of~\eqref{eq:evoProblem} on $[0,T]$ in the sense of Definition~\ref{defn:strongSolution}.
    \item Every strongly continuous evolution family is exponentially bounded.
  \end{enumerate}
\end{rem}

The following proposition states that the existence of an evolution family for $A$ already guarantees uniqueness of strong solutions for \eqref{eq:evoProblem}.
\begin{prop}[{\cite[Proposition~3.3.4]{Gallarati-17}, \cite[Proposition~4.5]{GallaratiV-18}}]
  \label{prop:uniqueness}
    Let $T > 0$ and $(U(t,s))_{0 \leq s \leq t \leq T}$ be an evolution family for $A$.
  \begin{enumerate}[(a)]
    \item If~\eqref{eq:evoProblem} has a strong solution $u \in \WW^{1,1}(0,T; X) \cap \Ell^1(0,T; D)$, then it satisfies
      \begin{align*}
          u(t) = U(t,s) u(s) \quad\text{for}\quad 0 \leq s \leq t \leq T  . 
      \end{align*}
      In particular, strong solutions are unique.
    \item If $(\widetilde{U}(t,s))_{0\leq s\leq t\leq T}$ is a further evolution family for $A$, then $\widetilde{U} = U$.
  \end{enumerate}
\end{prop}

\section{Observability for Evolution Families on Measurable Sets in Time}
\label{sec:obs}

In this section, we prove an abstract observability estimate for evolution families on Banach spaces formulated in Theorem~\ref{thm:observability}. Before we state that theorem, we introduce our main hypothesis.
\begin{hypo}
  \label{hypo:evoFam}
  Let $X$ and $Y$ be Banach spaces, $T>0$, and $(U(t,s))_{0\leq s\leq t \leq T}$ an exponentially bounded evolution family on~$X$.
  Let $C\colon [0,T]\to \cL(X,Y)$ be bounded,
  and assume that $[0,T] \ni t\mapsto \norm{C(t) U(t,0) x_0}_Y$
  is measurable for all $x_0\in X$.
Let $(P_\lambda)_{\lambda>0}$ in $\cL(X)$ be bounded.
Assume that there exist $d_0,d_1,\gamma_1 > 0$ such that 
\begin{align} 
  \forall \lambda > 0 \ \forall t\in[0,T] \ \forall x\in X &\colon \quad \lVert P_\lambda x \rVert_{ X } 
  \le 
   d_0 \euler^{d_1 \lambda^{\gamma_1}} \norm{ C(t)  P_\lambda x }_{Y}
        \label{eq:ass:uncertainty}
\end{align}
and $d_2\geq 1$ and $d_3,\gamma_2,\gamma_3 >0$ with $\gamma_1<\gamma_2$ such that
\begin{align}
\label{eq:ass:dissipation}
        \forall \lambda > 0 \ \forall 0\leq s\leq t\leq T \ \forall x_s\in X  &\colon \quad \lVert (\id-P_\lambda) U(t,s) x_s \rVert_{X} \le d_2 \euler^{-d_3 \lambda^{\gamma_2} (t-s)^{\gamma_3}} \lVert x_s \rVert_{X} .
\end{align}
\end{hypo}
\begin{rem}
    The estimate in \eqref{eq:ass:uncertainty} is an abstract uncertainty principle, while \eqref{eq:ass:dissipation} is a dissipation estimate.
\end{rem}

\begin{thm} \label{thm:observability} 
  Assume Hypothesis~\ref{hypo:evoFam} and let $E\subseteq [0,T]$ be measurable with positive Lebesgue measure.
  Then there exists $C_{\mathrm{obs}} \geq 0$ such that, for all $x_0\in X$ and $r\in[1,\infty]$, we have
	\begin{align*}
		\lVert U(T,0) x_0 \rVert_{X} \leq C_{\mathrm{obs}} \begin{cases} \Bigl(\int_E \norm{C(t)U(t,0) x_0}_Y^r\d t\Bigr)^{1/r}, & r\in [1,\infty),\\
		\esssup_{t\in E} \norm{C(t)U(t,0)x_0}_{Y}, & r=\infty .                
        \end{cases}
	\end{align*}
\end{thm}
If $E = [0,T]$, we give an explicit estimate on $C_{\mathrm{obs}}$ in Remark~\ref{rem:constants}.
For the proof of Theorem~\ref{thm:observability}, it is convenient to introduce the following shorthand notation: for all $x_0 \in X$, $0  \leq t \leq T$, and $\lambda > 0$, we define
\begin{align*}
  F(t) \coloneqq \|U(t,0)x_0 \|_X\,, \quad
  F_{\lambda}(t) &\coloneqq \|P_\lambda U(t,0)x_0 \|_X\,, \quad \\
  F_{\lambda}^\perp(t) &\coloneqq \|(\id - P_\lambda) U(t,0)x_0 \|_X
\end{align*}
and analogously for the observations
\begin{align*}
  G(t) \coloneqq \|C(t) U(t,0)x_0 \|_Y\,, \quad
  G_{\lambda}(t) &\coloneqq \|C(t) P_\lambda U(t,0)x_0 \|_Y\,, \quad \\
  G_{\lambda}^\perp(t) &\coloneqq \|C(t) (\id - P_\lambda) U(t,0)x_0 \|_Y\,.
\end{align*}

We start by recalling a standard interpolation argument formulated in the following lemma, cf.~\cite[p.~110]{Robbiano-95} or \cite[Lemma~5.2]{RousseauL-12} for similar statements.
\begin{lem} \label{lem:interpolation}
 Let $F_1,F_2,G , D , C \geq 0$, $\theta \in (0,1)$, and assume that, for all $\eps \in (0,1]$, we have
 \begin{align} \label{eq:assumption-lemma}
    F_2 \leq D F_1 
    \quad\text{and}\quad
    F_2 \leq C \left( \eps^{-\frac{\theta}{1 - \theta}} G + \eps F_1 \right) .
 \end{align}
 Then we have
 \[
  F_2 \leq \max \left\{  \frac{C}{\theta^{\theta}(1 - \theta)^{1-\theta}},  
                    D \Bigl( \frac{\theta}{1 - \theta}\Bigr)^{1- \theta} 
         \right\}  F_1^\theta G^{1-\theta} .
 \]
\end{lem}
\begin{proof} 
 If $F_1=0$ or $G=0$, the statement is obvious. (Note that, in case $G=0$, the second inequality in~\eqref{eq:assumption-lemma} yields $F_2=0$.) Therefore, let $F_1,G>0$.
  Then the right-hand side of the second inequality in~\eqref{eq:assumption-lemma} is minimal for 
  \begin{align*}
    \varepsilon_0 \coloneqq \left(\frac{\theta G }{ (1-\theta)F_1}\right)^{1-\theta} \; >0\,.
  \end{align*}
  \par
  If $\varepsilon_0\leq 1$, we have by assumption
      \begin{align*}
          F_2 & \leq 
          C \left( \eps_0^{-\frac{\theta}{1 - \theta}} G + \eps_0 F_1 \right)  =
			\frac{C}{\theta^{\theta}(1 - \theta)^{1-\theta}} F_1^{\theta} G^{1-\theta} .
		\end{align*}
  \par
    If $\varepsilon_0>1$, by the first inequality in~\eqref{eq:assumption-lemma} and the definition of $\varepsilon_0$, we observe
		\begin{align*}
			F_2
			& \leq D F_1^{\theta} F_1^{1-\theta} 
              = D F_1^{\theta} \left(\frac{\theta G}{(1-\theta)\varepsilon_0^{1/(1-\theta)}}\right)^{1-\theta} 
              < D\left(\frac{\theta}{1-\theta}\right)^{1-\theta} F_1^\theta G^{1-\theta}. \qedhere
		\end{align*}
\end{proof}
The following theorem is inspired by \cite[Theorem~1.2]{WangZ-17}, where a similar interpolation estimate is proven in the case of Hilbert spaces and one-parameter semigroups.
\begin{thm} \label{thm:interpolation} 
Assume Hypothesis~\ref{hypo:evoFam}, and let $\theta \in (0,1)$.
Then there exist $\widetilde C_1, \widetilde C_2, \widetilde{C}_3\geq 0$ such that, for all $x_0 \in X$ and $0\le  s < t \le T$, we have that
\begin{align*}
  \lVert U(t,0) x_0 \rVert_{X} 
  \leq 
  \widetilde C_1 \exp \left(\frac{\widetilde C_2}{(t-s)^{\frac{\gamma_1 \gamma_3}{\gamma_2 - \gamma_1}}} 
  +   \widetilde{C}_3 (t-s)\right) \lVert C(t) U(t,0)x_0 \rVert_{Y}^{1-\theta} \lVert U(s,0)x_0\rVert_{X}^{\theta}.
\end{align*}
\end{thm}
Explicit estimates on $\widetilde C_1$, $\widetilde C_2$, and $\widetilde C_3$ can be inferred from the proof and are summarized in Remark~\ref{rem:constants}.
\begin{proof}[Proof of Theorem~\ref{thm:interpolation}]
	Let $\theta \in (0,1)$, $x_0 \in X$, and $0 < t \le T$. Furthermore, let $\lambda > 0$. Using the uncertainty principle~\eqref{eq:ass:uncertainty}, we obtain
        \begin{align*}
          F(t) \le F_{\lambda}(t) + F_{\lambda}^\perp(t) \le d_0\euler^{d_1 \lambda^{\gamma_1}} G_{\lambda}(t) + F_{\lambda}^\perp(t) .
        \end{align*}
	Together with the estimate
        \begin{align*}
          G_{\lambda}(t) \le G(t) + G_{\lambda}^\perp(t) \le G(t) + \lVert C(t) \rVert_{\cL (X,Y)}F_{\lambda}^\perp(t)  ,
        \end{align*}
        we get
	\begin{equation}
          F(t) \le d_0\euler^{d_1 \lambda^{\gamma_1}} G(t) + \big(1 + d_0\lVert C(t) \rVert_{\cL (X,Y)}\big)\,\euler^{d_1 \lambda^{\gamma_1}}F_{\lambda}^\perp(t) .
		\label{eq:F1}
	\end{equation}
        Let $0\leq s< t$. By the algebraic property~\eqref{it:algebraic} of evolution families in Definition~\ref{defn:evoFam} and the dissipation estimate~\eqref{eq:ass:dissipation}, we obtain
        \begin{align*}
          F_{\lambda}^\perp(t) 
          = \bigl\lVert (\id - P_\lambda) U(t,s)U(s,0) x_0 \bigr\rVert_X 
          \leq d_2 \euler^{-d_3 \lambda^{\gamma_2} (t-s)^{\gamma_3}} F(s) .
        \end{align*}
        With~\eqref{eq:F1}, we have
	\begin{align}
	\begin{split}
		F(t) &\le d_0\euler^{d_1 \lambda^{\gamma_1}} G(t) + \big(1 + d_0\lVert C(t) \rVert_{\cL (X,Y)}\big)\, d_2\euler^{d_1 \lambda^{\gamma_1}-d_3 \lambda^{\gamma_2} (t-s)^{\gamma_3}} F(s)\\
		& \leq \widetilde{c}_1 \Bigl(\euler^{d_1 \lambda^{\gamma_1}} G(t) + \euler^{d_1 \lambda^{\gamma_1}-d_3 \lambda^{\gamma_2} (t-s)^{\gamma_3}} F(s)\Bigr),
    \end{split}
    \label{eq:F2}
	\end{align}
	where $\widetilde{c}_1\coloneqq \max\{d_0, (1+d_0\norm{C(\cdot)}_{\infty})d_2\} > 1$.
	
	Let now $f(\lambda) \coloneqq d_1 \lambda^{\gamma_1}-\theta d_3 \lambda^{\gamma_2} (t-s)^{\gamma_3}$.
	Then $f$ attains its maximum at the point $$\lambda^* \coloneqq \left(\frac{d_1\gamma_1}{\theta d_3\gamma_2}\right)^{\frac{1}{\gamma_2-\gamma_1}} \left(\frac{1}{t-s}\right)^{\frac{\gamma_3}{\gamma_2-\gamma_1}}.$$
	Thus,
        \begin{align*}
          f(\lambda)
          \leq f(\lambda^*) 
          = \left(\frac{d_1\gamma_1}{\theta d_3\gamma_2}\right)^{\frac{\gamma_1}{\gamma_2-\gamma_1}}d_1 \left (1-\frac{\gamma_1}{\gamma_2}\right) \left(\frac{1}{t-s}\right)^{\frac{\gamma_1\gamma_3}{\gamma_2-\gamma_1}} .
        \end{align*}
	This estimate and \eqref{eq:F2} imply that
	\begin{align*}
		F(t) 
                &\le \widetilde{c}_1 \exp\left(\widetilde{c}_2 \left(\frac{1}{t-s}\right)^{\frac{\gamma_1\gamma_3}{\gamma_2-\gamma_1}}\right) 
		\Bigl(\euler^{\theta d_3 \lambda^{\gamma_2}(t-s)^{\gamma_3}} G(t) + \euler^{-(1-\theta)d_3 \lambda^{\gamma_2}(t-s)^{\gamma_3}} F(s)\Bigr),
	\end{align*}
	where
	\[
	 \widetilde{c}_2 \coloneqq \left(\frac{d_1\gamma_1}{\theta d_3\gamma_2}\right)^{\frac{\gamma_1}{\gamma_2-\gamma_1}}d_1 \left(1-\frac{\gamma_1}{\gamma_2}\right) .
	\]
	As $\lambda > 0$ was arbitrary, we conclude that, for all $\eps \in (0,1]$, we have
\begin{equation} \label{eq:before-interpolation}
F(t) 
                \le \widetilde{c}_1 \exp\left(\widetilde{c}_2 \left(\frac{1}{t-s}\right)^{\frac{\gamma_1\gamma_3}{\gamma_2-\gamma_1}}\right) 
		\Bigl( \eps^{-\frac{\theta}{1-\theta}} G(t) + \eps F(s)\Bigr).
\end{equation}
        Since $(U(t,s))_{0\leq s\leq t\leq T}$ is exponentially bounded, there exist $M\geq 1$ and $\omega\in\R$ such that $F(t)\leq M\ee^{\omega (t-s)}F(s)$ for all $0\leq s\leq t\leq T$. Since $\widetilde c_1 \geq 1$, $M \geq 1$, and $(\theta^\theta (1-\theta)^{1-\theta})^{-1} \geq (\theta / (1-\theta))^{1-\theta}$ for all $\theta \in (0,1)$, we have
        \begin{multline*}
         \max \left\{ \frac{1}{\theta^{\theta}(1 - \theta)^{1-\theta}} \, \widetilde{c}_1  \exp\left(\widetilde{c}_2 \left(\frac{1}{t-s}\right)^{\frac{\gamma_1\gamma_3}{\gamma_2-\gamma_1}}\right)\,,\;
                    \Bigl(\frac{\theta}{1 - \theta}\Bigr)^{1- \theta} M \euler^{\omega (t-s)}
         \right\}
         \\
         \leq \frac{1}{\theta^{\theta}(1 - \theta)^{1-\theta}} M \,  \widetilde c_1   \exp\left(\widetilde{c}_2 \left(\frac{1}{t-s}\right)^{\frac{\gamma_1\gamma_3}{\gamma_2-\gamma_1}} +\omega_+ (t-s) \right) ,
        \end{multline*}
        where $\omega_+ \coloneqq \max\{\omega , 0\}$. Thus, setting $\widetilde{C}_1 \coloneqq (\theta^{\theta}(1 - \theta)^{1-\theta})^{-1} M \,  \widetilde{c}_1$, $\widetilde{C}_2 \coloneqq \widetilde{c}_2$, and $\widetilde{C}_3 \coloneqq \omega_+$, the statement of the theorem follows from \eqref{eq:before-interpolation} and Lemma~\ref{lem:interpolation}.
\end{proof}
\begin{rem}
For autonomous systems (i.e.\ $A(\cdot)$ and $C(\cdot)$ are time-independent) on Hilbert spaces with self-adjoint generator $A$ with purely discrete spectrum, the interpolation estimate in Theorem \ref{thm:interpolation} is equivalent to the uncertainty principle \eqref{eq:ass:uncertainty} with spectral projectors $P_\lambda$, see \cite[Theorem~2.1]{PhungWX-17}. Note that, in this case, \eqref{eq:ass:dissipation} is trivial. As was pointed out in \cite[Remark~2.2]{PhungWX-17}, the fact that the system is autonomous is crucial to obtain the equivalence.
\end{rem}
\begin{proof}[Proof of Theorem~\ref{thm:observability}]
By H\"older's inequality, it suffices to prove the case $r = 1$. Let $x_0\in X$. 
Since $(U(t,s))_{0\leq s\leq t\leq T}$ is exponentially bounded, there exist $M\geq 1$ and $\omega\in\R$ such that $F(\tau)\leq M\ee^{\omega (\tau-t)}F(t) \leq M \ee^{\omega_+ (\tau-t)}F(t)$ for all $0\leq t\leq \tau\leq T$, where $\omega_+ = \max\{\omega,0\}$.
Let $\theta\in (0,1)$. For all $0\leq s<t<\tau\leq T$, we apply Theorem \ref{thm:interpolation} and obtain
\begin{align} \label{eq:Ftau1}
  F(\tau)
  \leq 
  M \widetilde C_1 \exp \left(\frac{\widetilde C_2}{(t-s)^{\frac{\gamma_1 \gamma_3}{\gamma_2 - \gamma_1}}} +   \widetilde C_3 (\tau - s)\right) G(t)^{1-\theta} F(s)^{\theta} .
\end{align}
For the moment, let us fix $q \in (0,1)$, which will be specified below. Let $\ell$ be a Lebesgue point of $E$. By \cite[Proposition~2.1]{PhungW-13}, there exists a sequence~$(\ell_m)_{m \in \N}$ in $[0,T]$ with $\ell_m \to \ell$ such that, for all $m\in\N$, we have
  \begin{align*}
    \ell_m > \ell_{m+1}, \quad \ell_{m+1}-\ell_{m+2} = q \, (\ell_m - \ell_{m+1}), \quad\text{and}\quad \abs{E\cap (\ell_{m+1},\ell_m)} \geq \frac{\ell_{m}-\ell_{m+1}}{3} .
  \end{align*}
  We set $\xi_m \coloneqq \ell_{m+1} + (\ell_m - \ell_{m+1})/6$, $s=\ell_{m+1}$, and $\tau=\ell_m$. 
  Then, for $t \in (\xi_{m} , \ell_m)$, we have $t-s \geq \xi_m - \ell_{m+1} = (\ell_m - \ell_{m+1})/6$. Applying \eqref{eq:Ftau1}, we obtain for all $t \in (\xi_m, \ell_m)$ that
  \begin{align*}
  F(\ell_m)
  &\leq
  \widetilde{c}_1 \exp \left(\frac{\widetilde{c}_2}{\delta_m^{\frac{\gamma_1 \gamma_3}{\gamma_2 - \gamma_1}}} +   \widetilde C_3 \delta_m\right) G(t)^{1-\theta} F(\ell_{m+1})^{\theta} ,
\end{align*}
  where $\widetilde{c}_1 \coloneqq M \widetilde C_1$, $\widetilde{c}_2 \coloneqq 6^{\gamma_1 \gamma_3 / (\gamma_2 - \gamma_1)} \widetilde C_2$, and $\delta_m\coloneqq \ell_m-\ell_{m+1}$. Let $\eps>0$. 
  Then Young's inequality 
$
    ab
    \leq 
    \eps a^{1/\theta}  
    + 
    \eps^{-\theta / (1-\theta)} (1 - \theta) \theta ^{\theta  / (1 - \theta)} b^{1/(1 - \theta)} 
$
  with $a=F(\ell_{m+1})^\theta$ yields
	\[
		F(\ell_m) \leq \eps F(\ell_{m+1}) 
		+ \eps^{-\frac{\theta}{1-\theta}} (1 - \theta) \theta^{\frac{\theta }{1 - \theta }} \, \widetilde{c}_1^{\, \frac{1}{1-\theta}} \exp \left(\frac{\widetilde{c}_2}{(1-\theta)\delta_m^{\frac{\gamma_1 \gamma_3}{\gamma_2 - \gamma_1}}} +   \frac{\widetilde{C}_3}{1-\theta} \delta_m\right) 
        G(t) .
    \]
	Taking integral means (with respect to $t$) on $E\cap [\xi_m,\ell_m]$ and using that, by construction, we have $\lvert E\cap [\xi_m,\ell_m] \rvert \geq \delta_m/6$, we obtain
	\begin{multline*}
		F(\ell_m)   \leq \eps F(\ell_{m+1}) \\
		 + \eps^{-\frac{\theta}{1-\theta}} 
                (1 - \theta) \theta^{\frac{\theta }{1 - \theta }} \, 
                \widetilde{c}_1^{\, \frac{1}{1-\theta}} \exp \left(\frac{\widetilde{c}_2}{(1-\theta)\delta_m^{\frac{\gamma_1 \gamma_3}{\gamma_2 - \gamma_1}}} +   \frac{\widetilde{C}_3}{1-\theta} \delta_m\right)\frac{6}{\delta_m} \int_{\ell_{m+1}}^{\ell_m} \1_E(t)G(t)\d t
	\end{multline*}
	and therefore
	\begin{multline*}
          \eps^{\frac{\theta}{1-\theta}} \delta_m
          \exp\left({-\frac{\widetilde{c}_2}{(1-\theta)\delta_m^{\frac{\gamma_1 \gamma_3}{\gamma_2 - \gamma_1}}}}\right)  
          F(\ell_m)
          - 
          \eps^{\frac{1}{1-\theta}} \delta_m 
          \exp\left({-\frac{\widetilde{c}_2}{(1-\theta)\delta_m^{\frac{\gamma_1 \gamma_3}{\gamma_2 - \gamma_1}}}}\right)  F(\ell_{m+1}) \\
          \leq  
          (1 - \theta) \theta^{\frac{\theta }{1 - \theta }}
          \, \widetilde{c}_1^{\, \frac{1}{1-\theta}} 6 
          \exp\left({\frac{\widetilde{C}_3}{1-\theta} T}\right) \int_{\ell_{m+1}}^{\ell_m} \1_E(t)G(t)\d t.
	\end{multline*}
	Setting $\eps\coloneqq q\exp \bigl(-(1-\theta) / \delta_m^{\gamma_1 \gamma_3 / (\gamma_2 - \gamma_1)} \bigr)$ yields
	\begin{multline*}
		\delta_m 
                \exp\left({-\frac{\frac{\widetilde{c}_2 }{ 1-\theta} +\theta}{\delta_m^{\frac{\gamma_1 \gamma_3}{\gamma_2 - \gamma_1}}}}\right)  
                F(\ell_m) - 
                q\delta_m 
                \exp\left({-\frac{\frac{\widetilde{c}_2}{1-\theta} + 1}{\delta_m^{\frac{\gamma_1 \gamma_3}{\gamma_2 - \gamma_1}}}}\right)  
                F(\ell_{m+1}) \\
                \leq 
                q^{-\frac{\theta}{1-\theta}} 
                (1 - \theta) \theta^{\frac{\theta }{1 - \theta }}
                \, \widetilde{c}_1^{\, \frac{1}{1-\theta}} 6 
                \exp\left(\frac{\widetilde{C}_3}{1-\theta} T \right)
                \int_{\ell_{m+1}}^{\ell_m} \1_E(t)G(t)\d t.
	\end{multline*}
        Now we set $q \coloneqq \Bigl(\frac{\frac{\widetilde{c}_2}{1-\theta} + \theta}{\frac{\widetilde{c}_2}{1-\theta} + 1}\Bigr)^{\frac{\gamma_2 - \gamma_1}{\gamma_1 \gamma_3}}$. With this choice, we have
\begin{align*}
  \frac{\frac{\widetilde{c}_2}{1-\theta}+1}{\delta_m^{\frac{\gamma_1 \gamma_3}{\gamma_2 - \gamma_1}}}
  = 
  \frac{\frac{\widetilde{c}_2}{1-\theta} + \theta}{\frac{\widetilde{c}_2}{1-\theta} + 1}
  \;
  \frac{\frac{\widetilde{c}_2}{1-\theta}+1}{(q\delta_m)^{\frac{\gamma_1 \gamma_3}{\gamma_2 - \gamma_1}}}
  =
  \frac{\frac{\widetilde{c}_2}{1-\theta}+\theta}{\delta_{m+1}^{\frac{\gamma_1 \gamma_3}{\gamma_2 - \gamma_1}}}
\end{align*}
which leads us to the estimate
	\begin{align*}
          &\delta_m
          \exp\left({-\frac{\frac{\widetilde{c}_2}{1-\theta}+\theta}{\delta_m^{\frac{\gamma_1 \gamma_3}{\gamma_2 - \gamma_1}}}}\right)
          F(\ell_m) - 
          \delta_{m+1}
          \exp\left({-\frac{\frac{\widetilde{c}_2}{1-\theta}+\theta}{\delta_{m+1}^{\frac{\gamma_1 \gamma_3}{\gamma_2 - \gamma_1}}}}\right) 
          F(\ell_{m+1}) \\
         &\qquad 
          \leq  
          q^{-\frac{\theta}{1-\theta}} 
          (1 - \theta) \theta^{\frac{\theta }{1 - \theta }}
          \, \widetilde{c}_1^{\, \frac{1}{1-\theta}} 6 
          \exp\left({\frac{\widetilde{C}_3}{1-\theta} T}\right) 
          \int_{\ell_{m+1}}^{\ell_m} \1_E(t)G(t)\d t.
	\end{align*}
	Taking the sum over all $m\in\N$, a telescoping sum argument yields
        \begin{align*}
          &\delta_1
          \exp\left({-\frac{\frac{\widetilde{c}_2}{1-\theta}+\theta}{\delta_1^{\frac{\gamma_1 \gamma_3}{\gamma_2 - \gamma_1}}}}\right)  
          F(\ell_1) \\
          &\qquad
          \leq  
          q^{-\frac{\theta}{1-\theta}} 
          (1 - \theta) \theta^{\frac{\theta }{1 - \theta }}
          \, \widetilde{c}_1^{\, \frac{1}{1-\theta}} 6 
          \exp\left({\frac{\widetilde{C}_3}{1-\theta} T} \right)
          \int_{\ell}^{\ell_1} \1_E(t)G(t)\d t.
        \end{align*}
	Hence,
        \[
          F(\ell_1)
          \leq 
          q^{-\frac{\theta}{1-\theta}} 
          (1 - \theta) \theta^{\frac{\theta }{1 - \theta }}
          \, \widetilde{c}_1^{\, \frac{1}{1-\theta}} 6  
          \;
          \frac{\exp\left(
                \frac{\frac{\widetilde{c}_2}{1-\theta}+\theta}{(\ell_1-\ell_{2})^{\frac{\gamma_1 \gamma_3}{\gamma_2 - \gamma_1}}}
                + \frac{\widetilde{C}_3}{1-\theta} T\right)}{\ell_1-\ell_2} 
          \int_E G(t)\d t.
        \]
	Now, $F(T)\leq M\euler^{\omega (T-\ell_1)} F(\ell_1)$ yields the assertion.
\end{proof}

\begin{rem}
\label{rem:constants}
In this remark, we give explicit estimates on the constants appearing in Theorems~\ref{thm:interpolation} and \ref{thm:observability}.
    The proof of Theorem \ref{thm:interpolation} shows
    \begin{align*}
  \widetilde C_1 & = \frac{1}{\theta^{\theta}(1 - \theta)^{1-\theta}} \,M \max\big\{d_0, (1 + d_0 \norm{C(\cdot)}_\infty)d_2\big\},\\
\widetilde C_2 & = \Bigl(\frac{d_1\gamma_1}{\theta d_3\gamma_2}\Bigr)^{\frac{\gamma_1}{\gamma_2-\gamma_1}}d_1 \Bigl(1-\frac{\gamma_1}{\gamma_2}\Bigr),\\
\widetilde C_3 & = \omega_+.
\end{align*}
If $E=[0,T]$, then the constant $C_{\mathrm{obs}}$ in Theorem~\ref{thm:observability} can be made explicit as well. Indeed, choosing $\ell=0$ and $\ell_{m+1}=q^m T$ for $m\in\N_0$, we see that $\ell_1=T$ and $\ell_1-\ell_2 
  = (1-q)T$ and the proof of Theorem~\ref{thm:observability} yields the estimate
	\[C_{\mathrm{obs}}\leq \frac{{C}_1}{T^{1/r}} \exp \biggl(\frac{{C}_2}{T^{\frac{\gamma_1 \gamma_3}{\gamma_2 - \gamma_1}}} +   {C}_3 T\biggr) ,
	\]
	where $T^{1/\infty} \coloneqq 1$ and
	\begin{align*}
		{C}_1 
                & = q^{-\frac{\theta}{1-\theta}} (1 - \theta) \theta^{\frac{\theta }{1 - \theta }}
                M^{\frac{1}{1-\theta}} \widetilde C_1^{\frac{1}{1-\theta}} 
                \frac{6}{1-q} ,
         &        
		{C}_2 & = \frac{6^{\frac{\gamma_1 \gamma_3}{\gamma_2 - \gamma_1}}\frac{\widetilde C_2}{1-\theta}+\theta}{(1-q)^{\frac{\gamma_1\gamma_3}{\gamma_2-\gamma_1}}}, 
		\\
		{C}_3 & = \frac{\omega_+}{1-\theta},
		 & 
		q &= \Biggl(\frac{6^{\frac{\gamma_1 \gamma_3}{\gamma_2 - \gamma_1}}\frac{\widetilde C_2}{1-\theta} + \theta}{6^{\frac{\gamma_1 \gamma_3}{\gamma_2 - \gamma_1}}\frac{\widetilde C_2}{1-\theta} + 1}\Biggr)^{\frac{\gamma_2 - \gamma_1}{\gamma_1 \gamma_3}} ,
    \end{align*}
	with $\theta \in (0,1)$. In particular, Theorem~\ref{thm:observability} yields a non-autonomous version of the results in \cite[Theorem~2.1]{GallaunST-20} and \cite[Theorem~A.1]{BombachGST-20}.
\end{rem}

\section{Observability for Non-autonomous Elliptic Operators}\label{sec:elliptic}
In this section, we apply the preceding theory to the evolution family associated with a non-autonomous parabolic equation on the domain $\R^d$. 

\subsection{Non-autonomous Elliptic Operators}
\label{subsec:na_elliptic_ops}
As a preparation, let us introduce the operators and notions used throughout this section.
Our aim is to define a family of non-autonomous differential operators and derive $\Ell^p$-bounds for the associated evolution family.
Let $T>0$.

\begin{defn}[Non-autonomous Elliptic Polynomial]
	\label{defn:ellipticPolynomial}
	Let $m \in \N$. For $\alpha \in \N_0^d$ with $\abs{\alpha}\leq m$, let $a_\alpha\from [0,T] \to \C$. Then we call $\fa\from [0,T] \times \R^d\to \C$ given by
	\begin{align*}
		\fa(t,\xi)  \coloneqq  \sum_{|\alpha| \leq m} a_\alpha(t) (\ii \xi)^\alpha ,\quad t \in [0,T],~\xi \in \R^d,
	\end{align*}
	\emph{non-autonomous polynomial of degree $m$}.
	The \emph{principal symbol} of $\fa$ is given by
	\[\fa_m(t,\xi)  \coloneqq  \sum_{|\alpha| = m} a_\alpha(t) (\ii \xi)^\alpha ,\quad t \in [0,T],~\xi \in \R^d.\]
	We call $\fa$ \emph{uniformly strongly elliptic (with respect to $t$)} if there exists $c > 0$ such that, for all $t\in [0,T]$ and $\xi \in \R^d$, we have
	\begin{align}
		\label{eq:ellipticPolynomial}
		\Re \fa_m(t,\xi) \geq c |\xi|^m .
	\end{align}
\end{defn}
Note that uniform strong ellipticity implies that $m$ is even such that we will always have $m \geq 2$ for the degree of the non-autonomous polynomial $\fa$.

We define the Fourier transformation $\cF \colon \cS(\R^d) \to \cS(\R^d)$ on the \emph{Schwartz space} by
\[
  (\cF u)(\xi) \coloneqq \int_{\R^d} \ee^{-\ii x \cdot \xi} u (x) \d x ,\quad \xi \in \R^d.
\]
As usual, we extend $\cF$ and its inverse $\cF^{-1}$ to automorphisms of the space of \emph{tempered distributions} $\cS'(\R^d)$.

Using (uniformly strongly elliptic) non-autonomous polynomials as symbols of Fourier multipliers gives rise to a certain class of differential operators.

\begin{defn}[Elliptic Operator]
  Let $\fa$ be a non-autonomous polynomial of degree $m \geq 2$. 
  For $t\in [0,T]$, we define $A(t)\from \cS'(\R^d)\to \cS'(\R^d)$ by
  \[A(t)u \coloneqq \cF^{-1}(\fa(t,\cdot) \cF u) = \sum_{|\alpha| \leq m} a_\alpha(t) \partial^\alpha u.\]
  We call the family $(A(t))_{t\in[0,T]}$ the \emph{operator family associated with $\fa$}.
  If, furthermore, $\fa$ is uniformly strongly elliptic, we call $(A(t))_{t\in[0,T]}$ \emph{elliptic}. 
\end{defn}
Let $\fa$ be uniformly strongly elliptic and $(A(t))_{t \in [0,T]}$ the associated family of elliptic operators.
Note that $A(t)$ leaves $\cS(\R^d)$ invariant for all $t \in [0,T]$.
Moreover, for $p \in [1,\infty)$ and $t\in[0,T]$, the part $A_p(t)$ of $A(t)$ in $X \coloneqq \Ell^p(\R^d)$ is a closed and densely defined operator with $D(A_p(t)) = \WW^{p,m}(\R^d)$ for $p>1$, while only $\WW^{1,m}(\R^d)\subseteq D(A_1(t))$, see \cite[Chapter~8]{Haase-06}. Thus, let us denote $D^p \coloneqq D(A_p(t))$ for $t \in [0,T]$, noting that $D(A_p(t))$ does not depend on $t$. 
Furthermore, $\cS(\R^d)$ is dense in $D^p$ with respect to the graph norm of $A_p(t)$ for all $t \in [0,T]$.
Moreover, in case $p=\infty$, for $t\in[0,T]$, we set $A_\infty(t) \coloneqq \widetilde{A}_1(t)'$, where $(\widetilde{A}_1(t))_{t\in[0,T]}$ is the operator family on $\Ell^1(\R^d)$ associated with the uniformly strongly elliptic polynomial $\widetilde{\fa} \coloneqq \fa(\cdot,-\cdot)$.

For $p \in [1,\infty)$, we can associate the non-autonomous Cauchy problem
\begin{align*}
  \dot{u}(t) & = -A_p(t) u(t), \quad t \in (0,T],  \quad u(0) = u_0\in \Ell^p(\R^d)
\end{align*}
to the operator family $(A_p(t))_{t \in [0,T]}$.
Under certain conditions on the coefficients of~$\fa$, we will define a (strongly continuous) evolution family for~$(A_p(t))_{t \in [0,T]}$.

Let $\fa$ be a uniformly strongly elliptic polynomial of degree $m \geq 2$ with coefficients $a_\alpha \in \Ell^1(0,T)$ for $\abs{\alpha}\leq m$. 
Then, as a consequence of the ellipticity estimate~\eqref{eq:ellipticPolynomial}, for $0\leq s < t \leq T$, we have 
\begin{align*}
  \ee^{-\int\limits_s^t \fa(\tau,\cdot)\d \tau}\in \cS(\R^d) .
\end{align*}
Thus, for $0\leq s\leq t \leq T$, we define $U(t,s)\from \cS'(\R^d)\to \cS'(\R^d)$ by
\begin{align}\label{eq:defnEvoFam}
  U(s,s)u \coloneqq u,\qquad U(t,s) u \coloneqq \cF^{-1}\Bigl(\ee^{-\int\limits_s^t \fa(\tau,\cdot)\d \tau} \cF u\Bigr) ,\quad t>s .
\end{align}
It is easy to see that, for $0\leq s<t \leq T$, the operator $U(t,s)$ is given as a convolution operator with kernel $p_{t,s}\in \cS(\R^d)$ defined via
\begin{align}\label{eq:kernel}
	p_{t,s} \coloneqq 
		\cF^{-1}\ee^{-\int\limits_s^t \fa(\tau,\cdot)\d \tau} .
\end{align}

The next lemma collects several algebraic properties of $(U(t,s))_{0\leq s\leq t\leq T}$.

\begin{lem}\label{lem:evoFam}
	Let $(U(t,s))_{0\leq s \leq t \leq T}$ be the operator family defined in~\eqref{eq:defnEvoFam}.
        \begin{enumerate}[(a)]
		\item\label{it:evoAlgebraic} For $0\leq r \leq s \leq t \leq T$, we have that
		\begin{align*}
			U(s,s) = \id \quad\text{and}\quad U(t,r) = U(t,s) U(s,r) .
		\end{align*}
		Moreover, $p_{t,r} = p_{t,s} \ast p_{s,r}$ for all $0\leq r<s<t \leq T$.
	      \item\label{it:evoOnLp} For $p \in [1,\infty]$ and $0\leq s\leq t \leq T$, $U(t,s)$ leaves $\Ell^p(\R^d)$ invariant.
    \end{enumerate}
\end{lem}

\begin{proof}
    The proof of~\eqref{it:evoAlgebraic} is straightforward from the definitions of the operator family $(U(t,s))_{0\leq s\leq t\leq T}$ in~\eqref{eq:defnEvoFam} and of its kernel in~\eqref{eq:kernel}. 
  Statement~\eqref{it:evoOnLp} is a consequence of Young's inequality, see, e.g.,~\cite[Theorem~1.2.10]{Grafakos-14}.
\end{proof}

Let~$(U(t,s))_{0\leq s\leq t \leq T}$ be as in~\eqref{eq:defnEvoFam} and $p\in [1,\infty]$. 
For $0\leq s\leq t \leq T$, we define $U_p(t,s) \coloneqq U(t,s)|_{\Ell^p(\R^d)}$. 
By Lemma~\ref{lem:evoFam}, $U_p(t,s)$ is a bounded operator on~$\Ell^p(\R^d)$ with $\|U_p(t,s)\|_{\cL(\Ell^p(\R^d))} = \|p_{t,s}\|_{\Ell^1(\R^d)}$ for $0\leq s<t \leq T$. 
Thus, $(U_p(t,s))_{0\leq s\leq t\leq T}$ is an evolution family on $\Ell^p(\R^d)$ in the sense of Definition~\ref{defn:evoFam}\eqref{it:algebraic}.

Under suitable assumptions on the coefficients $a_\alpha$, it is possible to show that the evolution family $(U_p(t,s))_{0\leq s\leq t\leq T}$ is strongly continuous and exponentially bounded. In fact, the evolution family $(U_p(t,s))_{0\leq s\leq t\leq T}$ can be seen as the solution operator to the non-autonomous Cauchy problem~\eqref{eq:evoProblem} for $(A_p(t))_{t \in [0,T]}$. 
The proof of these facts is postponed to Appendix~\ref{sec:properties}. 
The following theorem summarizes all of these properties.

\begin{thm}\label{thm:evoFamForElliptic}
	Let $\fa$ be a uniformly strongly elliptic polynomial of degree $m \geq 2$ with coefficients $a_\alpha \in \Ell^\infty(0,T)$ for $|\alpha| \leq m$.
	Let $(U(t,s))_{0\leq s \leq t \leq T}$ be defined as in~\eqref{eq:defnEvoFam}.
	\begin{enumerate}[(a)]
            \item\label{it:expEvoFam} Let $p \in [1,\infty]$. Then $(U_p(t,s))_{0\leq s \leq t \leq T}$ is an exponentially bounded evolution family. 
	\item Let $p\in (1,\infty)$. Then $(U_p(t,s))_{0\leq s \leq t \leq T}$ is 
	the unique evolution family for the family of operators $(A_p(t))_{t \in [0,T]}$.
	\end{enumerate}
\end{thm}

\begin{proof}
  By Lemma~\ref{lem:evoFam}, $(U_p(t,s))_{0\leq s\leq t \leq T}$ is an evolution family and Lemma~\ref{lem:kernel_bound} yields the exponential bound.
    
  Moreover, Proposition~\ref{prop:evoFamBounded} yields that $(U_p(t,s))_{0\leq s\leq t \leq T}$ is an evolution family for $(A_p(t))_{t \in [0,T]}$ in case $p\in(1,\infty)$. Uniqueness follows from Proposition~\ref{prop:uniqueness}.
\end{proof}

\subsection{Observability}

In this subsection, we show an observability estimate for the evolution family $(U_p(s,t))_{0 \leq s \leq t \leq T}$ from Subsection~\ref{subsec:na_elliptic_ops}.
For this purpose, we introduce the notion of a \emph{thick} subset $\thickset$ of $\R^d$.
Loosely speaking, a thick subset is a set such that the portion of it in a hypercube is bounded away from zero no matter where the hypercube is located.
In the following, given a measurable set $\Omega \subseteq \R^d$, let $|\Omega|$ denote its Lebesgue measure.

\begin{defn}[Thick Set]
	Let $L \in (0,\infty)^d$ and $\rho > 0$. 
        \begin{enumerate}[(a)]
		\item A set $\thickset \subseteq \R^d$ is called \emph{$(L,\rho)$-thick} if $\thickset$ is measurable and, for all $x \in \R^d$, we have
		\[
		\left\lvert \thickset \cap \left( \bigtimes_{i = 1}^d (0, L_i) + x \right) \right\rvert
		\geq \rho \prod_{i=1}^d L_i .
		\]
		\item Let $T>0$. A family $(\thickset (t))_{t \in [0,T]}$ of sets $\thickset (t) \subseteq \R^d$ is called \emph{mean $(L,\rho)$-thick on $[0,T]$} if $\thickset (t)$ is measurable for all $t \in [0,T]$, the mapping $[0,T]\times \R^d \ni (t,x)\mapsto \1_{\thickset(t)}(x)$ is measurable, and, for all $x \in \R^d$, we have
		\[
		\frac{1}{T}\int_0^T \left\lvert \thickset (t) \cap \left( \bigtimes_{i = 1}^d (0, L_i) + x \right) \right\rvert \d t
		\geq \rho \prod_{i=1}^d L_i .
		\]
		\item Let $T>0$. A family $(\thickset (t))_{t \in [0,T]}$ of sets $\thickset (t) \subseteq \R^d$ is called \emph{uniformly $(L,\rho)$-thick on $[0,T]$} if $\thickset (t)$ is $(L,\rho)$-thick for all $t \in [0,T]$ and the mapping $[0,T]\times \R^d \ni (t,x)\mapsto \1_{\thickset(t)}(x)$ is measurable.
	\end{enumerate}
	We call $\thickset \subseteq \R^d$ \emph{thick} if there exist $L \in (0,\infty)^d$ and $\rho > 0$ such that $\thickset$ is $(L,\rho)$-thick.
        Likewise, $(\Omega(t))_{t \in [0,T]}$ is called \emph{mean/uniformly thick} if it is mean/uniformly $(L, \rho)$-thick on $[0,T]$ for some $L \in (0,\infty)^d$ and $\rho > 0$.
\end{defn}
Note that equivalent notions of (mean/uniform) thickness are obtained by replacing the hypercubes $\bigtimes_{i = 1}^d (0,L_i)$ with balls $\BB(0,R)$ with some radius $R>0$.

\begin{example}
    Let $\thickset_1 = [0,\infty)$, $\thickset_2 = (-\infty, 0]$, $T = 2$, and
    \begin{align*}
      \thickset(t) &\coloneqq \begin{cases}
        \thickset_1, \quad  t\in[0,1),  \\
        \thickset_2, \quad  t\in [1,2] . 
      \end{cases}
    \end{align*}
      Then $(\thickset(t))_{t\in[0,T]}$ is mean $(L,1/2)$-thick for all $L>0$ but not uniformly thick.
\end{example}

\begin{lem}\label{lem:measurableNorm}
    Let $\fa$ be a uniformly strongly elliptic polynomial of degree $m \geq 2$ with coefficients $a_\alpha \in \Ell^\infty(0,T)$ for $|\alpha| \leq m$,
	and let $(U(t,s))_{0\leq s \leq t \leq T}$ be defined as in~\eqref{eq:defnEvoFam}.
    For each $t\in[0,T]$, let $\thickset(t)\subseteq\R^d$ be measurable, and assume that $[0,T]\times \R^d \ni (t,x)\mapsto \1_{\thickset(t)}(x)$ is measurable. Let $p \in [1,\infty]$ and $u_0\in \Ell^p(\R^d)$. Then $[0,T]\ni t\mapsto \norm{\1_{\thickset(t)} U_p(t,0)u_0}_{\Ell^p(\R^d)}$ is measurable.
\end{lem}

\begin{proof}
  By Corollary~\ref{cor:strong_continuity}, $(U_p(t,s))_{0\leq s\leq t \leq T}$ is strongly continuous for $p\in [1,\infty)$ and strongly continuous w.r.t.\ the weak$^*$-topology for $p=\infty$. 
  For $p\in[1,\infty)$, this implies directly the measurability of $[0,T]\ni t\mapsto \norm{\1_{\thickset(t)} U_p(t,0)u_0}_{\Ell^p(\R^d)}$. For $p=\infty$, the measurability follows from the variational description of the $\Ell^\infty$-norm via the canonical pairing with $\Ell^1$-elements and the strong continuity of $(U_\infty(t,s))_{0\leq s\leq t \leq T}$ w.r.t.\ the topology $\sigma(\Ell^\infty(\R^d),\Ell^1(\R^d))$.
\end{proof}

Our first result shows that uniform thickness implies an observability estimate.

\begin{thm}\label{thm:obs}
	Let $\fa$ be a uniformly strongly elliptic polynomial of degree $m \geq 2$ with coefficients $a_\alpha\in \Ell^\infty(0,T)$ for $\abs{\alpha}\leq m$. 
        Let $(U (t,s))_{0\leq s\leq t \leq T}$ be as in~\eqref{eq:defnEvoFam}.
	Let $(\thickset (t))_{t \in [0,T]}$ be uniformly thick on $[0,T]$. 
	Let $E \subseteq [0,T]$ be measurable with positive Lebesgue measure and $r\in [1,\infty]$. Then there exists $C_{\mathrm{obs}}\geq 0$ such that, for all $p\in[1,\infty]$ and $u_0\in \Ell^p(\R^d)$, we have
	\[\norm{U_p(T,0)u_0}_{\Ell^p(\R^d)} \leq C_{\mathrm{obs}}\begin{cases}
                                            \Bigl(\int_{E} \norm{(U_p(t,0)u_0)|_{\thickset(t)}}_{\Ell^p(\thickset(t))}^r \d t\Bigr)^{1/r}, & r\in [1,\infty),\\
                                            \esssup_{t\in E} \norm{(U_p(t,0)u_0)|_{\thickset(t)}}_{\Ell^p(\thickset(t))}, & r=\infty.
	                                     \end{cases}\]
\end{thm}

\begin{rem}
    In the situation of Theorem \ref{thm:obs}, if $E=[0,T]$, then we obtain
    \[C_{\mathrm{obs}} \leq \frac{C_1}{T^{1/r}}\exp\Bigl(\frac{C_2}{T^{\frac{\gamma_1\gamma_3}{\gamma_2-\gamma_1}}}+C_3 T\Bigr)\]
    for some $C_1, C_2, C_3\geq 0$, $\gamma_1 = \gamma_3 = 1$, and $\gamma_2 = m$; cf.\ Remark \ref{rem:constants}.
\end{rem}

\begin{proof}[Proof of Theorem~\ref{thm:obs}]
This proof consists of two parts. In the first part, we will show a dissipativity estimate, and, in the second part, we will derive an abstract uncertainty estimate. 
As both estimates do not depend on the value of $p$, it follows from Theorem~\ref{thm:observability} that also the observability constant $C_{\mathrm{obs}}$ can be chosen independently of $p$. 

We start by introducing a family of smooth frequency cutoffs.
To this end, let $\eta\in \CC_{\mathrm c}^\infty ([0,\infty) )$ with $0\leq\eta\leq 1$ such that $\eta (r) = 1$ for $r\in [0,1/2]$ and $\eta (r) = 0$ for $r\geq 1$. 
    For $\lambda > 0$, we define $\chi_\lambda\from \R^d\to \R$ by $\chi_\lambda (\xi) \coloneqq \eta (\lvert \xi \rvert / \lambda)$. Since $\chi_\lambda \in \mathcal{S}(\R^d)$ for all $\lambda >0$, we have $\mathcal{F}^{-1}\chi_\lambda \in \mathcal{S}(\R^d)$. For $\lambda > 0$, we define $P_\lambda \from  \Ell^p(\R^d) \to \Ell^p(\R^d)$ by $P_\lambda f \coloneqq (\mathcal{F}^{-1} \chi_\lambda) \ast f$.
    Then, for all $\lambda > 0$, the operator $P_\lambda$ is a bounded linear operator, the family $(P_\lambda)_{\lambda>0}$ is uniformly bounded by $\lVert \mathcal{F}^{-1} \chi_1 \rVert_{\Ell^{1}(\R^d)}$, and, for all $f\in \mathcal{S}(\R^d)$, we have $P_\lambda f\in\mathcal{S}(\R^d)$, $\cF P_\lambda f = \chi_\lambda \cF f \in \mathcal{S}(\R^d)$, and $\supp \cF P_\lambda f \subseteq \{y \in \R^d \colon \lvert y \rvert \leq \lambda\} \subseteq [-\lambda, \lambda]^d$, see \cite[Theorem~3.3]{GallaunST-20} for details.
	
	Since $\fa$ is uniformly strongly elliptic, there exists $c > 0$ such that, for all $t \in [0,T]$ and all $\xi \in \R^d$, we have $\re \fa_m (t,\xi) \geq c \lvert \xi \rvert^m$. 
        We define the (autonomous) uniformly strongly elliptic polynomials $\fb, \widetilde{\fb} \from [0,T] \times \R^d \to \C$ by
	\[\fb(t,\xi) \coloneqq \lvert \xi\rvert^m,\quad \widetilde{\fb} (t , \xi) \coloneqq \frac{c}{2}\lvert \xi \rvert^m,\]
	and set $\widetilde{\fa} \coloneqq \fa - \widetilde{\fb}$. Note that $\widetilde{\fa}$ is also uniformly strongly elliptic.
	
        Let $(V(t,s))_{0\leq s\leq t \leq T}$, $(\widetilde{V}(t,s))_{0\leq s\leq t \leq T}$, and $(\widetilde{U}(t,s))_{0\leq s\leq t \leq T}$ be as in \eqref{eq:defnEvoFam} for $\fb$, $\widetilde{\fb}$, and $\widetilde{\fa}$, respectively. Note that $\widetilde{V}(t,s) = V(\frac{c}{2}t,\frac{c}{2}s)$ for all $0\leq s\leq t\leq T$.
	
	Let $p\in [1,\infty]$.
        For $f\in \Ell^p(\R^d)$ and $0\leq s\leq t \leq T$, we have by definition
	\begin{align*}
		U_p(t,s)f & = \cF^{-1} \Bigl(\ee^{-\int\limits_s^t (\widetilde{\fb}(\tau,\cdot) + \widetilde{\fa}(\tau,\cdot) ) \d \tau} \cF f\Bigr)\\
		& = \cF^{-1} \Bigl(\ee^{-\int\limits_s^t \widetilde{\fb}(\tau,\cdot) \d \tau}  \cF \cF^{-1}\Bigl(\ee^{-\int\limits_s^t \widetilde{\fa}(\tau,\cdot)\d \tau} \cF f\Bigr)\Bigr) 
		= \widetilde{V}_p(t,s) \widetilde{U}_p(t,s) f.
	\end{align*}
	
	By \cite[Proposition~3.2]{BombachGST-20}, we infer that there exists $K_{m,d}\geq 0$, depending only on $m$ and $d$, such that, for all $\lambda>0$, all $f\in  \Ell^p(\R^d)$, and $0\leq s\leq t \leq T$, we have
	\[
	\norm{(\id-P_\lambda)V_p(t,s) f}_{\Ell^p(\R^d)} \leq K_{m,d}\, \ee^{-2^{-m-3}(t-s) \lambda^m} \norm{f}_{\Ell^p(\R^d)}.
	\]
	Thus, we also conclude
	\[
	\norm{(\id-P_\lambda)\widetilde{V}_p(t,s) f}_{\Ell^p(\R^d)} \leq K_{m,d}\, \ee^{-2^{-m-3}c/2 (t-s) \lambda^m} \norm{f}_{\Ell^p(\R^d)}.
	\]
        Moreover, by Theorem \ref{thm:evoFamForElliptic} there exist $\widetilde M \geq 1$ and $\widetilde \omega \in \R$, depending on $\widetilde{\fa}$ and therefore on $\fa$, such that $\norm{\widetilde{U}_p(t,s)}_{\cL(\Ell^p(\R^d))} \leq \widetilde M \ee^{\widetilde \omega(t-s)}$ for all $0\leq s\leq t \leq T$. Note that we can choose $\widetilde{\omega}=\omega$, where $\omega$ is an exponential growth rate for $(U_p(t,s))_{0\leq s\leq t\leq T}$ (by choosing the same $c_0$ in Lemma \ref{lem:real_part_bound} and inspecting the proof of Lemma \ref{lem:kernel_bound}). Thus, for $\lambda>\lambda^* \coloneqq (2^{m+5}\max\{\omega,0\}/c)^{1/m}$, $f\in\Ell^p(\R^d)$,
        and $0\leq s\leq t\leq T$, we arrive at
	\begin{align}
	\begin{split}
		\norm{(\id-P_\lambda)U_p(t,s) f}_{\Ell^p(\R^d)} 
    & = \norm{(\id-P_\lambda)\widetilde{V}_p(t,s)\widetilde{U}_p(t,s) f}_{\Ell^p(\R^d)} \\
		&\leq K_{m,d}\, \ee^{-2^{-m-3}c/2(t-s)\lambda^m}  \widetilde{M}\ee^{\omega (t-s)} \norm{f}_{\Ell^p(\R^d)} \\
		&\leq K_{m,d}\,\widetilde{M} \ee^{-(t-s)2^{-m-5} c \lambda^m} \norm{f}_{\Ell^p(\R^d)} .
    \end{split}
    \label{eq:Dissipation-elliptic}
	\end{align}
	
        Let $L \in (0,\infty)^d$ and $\rho > 0$ such that $(\Omega(t))_{t \in [0,T]}$ is uniformly $(L,\rho)$-thick, $f\in\Ell^p(\R^d)$, and $\lambda >0$. Since $\supp \cF P_\lambda f \subseteq [-\lambda, \lambda]^d$, the Logvinenko--Sereda theorem \cite[Theorem~3]{Kovrijkine-01} implies
	\begin{align}
		\|P_\lambda f\|_{\Ell^p(\R^d)} 
		\leq \ee^{-K d \ln(\rho/K^d)} \ee^{-2 K |L|_1 \ln\left( \rho /K^d \right) \lambda} \|(P_\lambda f)|_{\thickset(t)}\|_{\Ell^p(\thickset(t))}
		\label{eq:d0d1logvinenkoSereda}    
	\end{align}
	for all $t\in [0,T]$, where $K\geq0$ is a universal constant.
        By  \eqref{eq:Dissipation-elliptic}, \eqref{eq:d0d1logvinenkoSereda}, Theorem~\ref{thm:evoFamForElliptic}\eqref{it:expEvoFam}, and Lemma~\ref{lem:measurableNorm}, we conclude that Hypothesis \ref{hypo:evoFam} is satisfied with $Y = \Ell^p(\R^d)$ and $C(t)$ the restriction operator on $\Omega(t)$ for $t \in [0,T]$.
        Therefore, Theorem~\ref{thm:observability} yields the assertion.
\end{proof}

The following theorem will show a partial converse of Theorem~\ref{thm:obs}, namely that a final-state observability estimate implies that the family $(\thickset(t))_{t\in [0,T]}$ is mean thick. 
For the pure Laplacian on $\Ell^2(\R^d)$ and time-independent set of observability, such a result has first been shown in \cite{EgidiV-18,WangWZZ-19}. 
In the autonomous case, this has been generalized to strongly elliptic operators in $\Ell^p(\R^d)$ in \cite[Theorem~3.3]{GallaunST-20}. We also refer to \cite[Theorem~5]{BeauchardEP-S-20} where a similar result is shown for the non-autonomous Ornstein--Uhlenbeck equation.

\begin{thm}
  \label{thm:obsImpliesMeanThick}
	Let $\fa$ be a uniformly strongly elliptic polynomial of degree $m \geq 2$ with coefficients $a_\alpha\in \Ell^\infty(0,T)$ for $\abs{\alpha}\leq m$. 
    Let $(\thickset(t))_{t \in [0,T]}$ be such that $\thickset(t) \subseteq \R^d$ is measurable for all $t \in [0,T]$ and $[0,T]\times\R^d\ni(t,x)\mapsto\1_{\thickset(t)}(x)$ is measurable. Let $(U (t,s))_{0\leq s\leq t \leq T}$ be as in \eqref{eq:defnEvoFam}.
        Let $p, r\in[1,\infty)$, and assume there exists $C_\obs \geq 0$ such that, for all $u_0 \in \Ell^p(\R^d)$, we have
    \[\|U_p(T,0) u_0 \|_{\Ell^p(\R^d)} \leq C_\obs
                                            \Bigl(\int_{0}^T \norm{U_p(t,0)u_0|_{\thickset(t)}}_{\Ell^p(\thickset(t))}^r \d t\Bigr)^{1/r}.\]
	Then the family $(\thickset(t))_{t\in[0,T]}$ is mean thick.
\end{thm}

\begin{proof}
  Our proof is inspired by \cite{EgidiV-18,WangWZZ-19,BeauchardEP-S-20}.
  We will show the contrapositive: assume that the family $(\thickset(t))_{t\in[0,T]}$ is not mean thick.
    Then there exists a sequence $(x_n)_{n \in \N}$ in $\R^d$ such that, for all $n \in \N$, we have
  \begin{align}
    \label{eq:notMeanThick}
      \frac{1}{T} \int_0^T 
      |\thickset(t) \cap \BB(x_n, n) |^{r/p} \d t
        < \frac{1}{n} .
  \end{align}
  Let $f\in\cS(\R^d)$, $\norm{f}_{\Ell^p(\R^d)} = 1$, and set $f_n \coloneqq f(\cdot-x_n)$ for $n\in\N$.
  Let $t \in (0,T)$ and $n\in\N$. Then $U_p(t,0)f_n = p_{t,0}\ast f_n = p_{t,0}\ast f(\cdot-x_n)$. Moreover,
  \begin{align}
    \label{eq:thickU}
    \begin{split}
    & \|(U_p(t,0) f_n)|_{\thickset(t)} \|_{\Ell^p(\thickset(t))}^p
    = \| \one_{\thickset(t)} U_p(t,0) f_n\|_{\Ell^p(\R^d)}^p\\
    & =\| \one_{\thickset(t)} p_{t,0}\ast f(\cdot - x_n) \|_{\Ell^p(\R^d)}^p 
    =\| \one_{\thickset(t) - x_n} p_{t,0}\ast f \|_{\Ell^p(\R^d)}^p \\
    & = \| \one_{(\thickset(t) - x_n) \cap \BB(0,n)} p_{t,0}\ast f \|_{\Ell^p(\R^d)}^p
    + \| \one_{(\thickset(t) - x_n)} (1 - \one_{\BB(0,n)}) p_{t,0} \ast f \|_{\Ell^p(\R^d)}^p .
    \end{split}
  \end{align}
  
  We first estimate the first summand on the right-hand side of \eqref{eq:thickU}. 
  As a consequence of Lemma \ref{lem:kernel_bound}, there exists $C\geq 0$ such that $\norm{p_{t,0}}_{\Ell^{p'}(\R^d)}^{p}\leq C$ for all $t\in (0,T)$, where $\frac{1}{p'} + \frac{1}{p} = 1$. 
  By H\"older's and Young's inequality, we estimate
  \begin{align*}
    \begin{split}
    \| \one_{(\thickset(t) - x_n) \cap \BB(0,n)} p_{t,0}\ast f \|_{\Ell^p(\R^d)}^p
    &\leq |(\thickset(t) - x_n) \cap \BB(0,n)| \, \norm{p_{t,0}}_{\Ell^{p'}(\R^d)}^{p}\\
    &\leq  C |\thickset(t) \cap \BB(x_n,n)|.
    \end{split}
  \end{align*}
 
    For the second summand on the right-hand side of \eqref{eq:thickU}, we have by Lemma \ref{lem:kernel_bound}, H\"older's inequality, and Fubini--Tonelli's theorem
  \begin{align*}
    &\| \one_{(\thickset(t) - x_n)} (1 - \one_{\BB(0,n)}) p_{t,0}\ast f \|_{\Ell^p(\R^d)}^p
    \leq 
    \| (1 - \one_{\BB(0,n)}) p_{t, 0}\ast f \|_{\Ell^p(\R^d)}^p\\
    & \leq \int_{\complement \BB(0,n)} \! \Bigl(\int_{\R^d} C_1\ee^{\omega t} \ee^{-C_2 \abs{z}^{m/(m-1)}} \abs{f(x-t^{1/m}z)} \d z\Bigr)^p \d x \\
    & = C_1^p\ee^{p\omega_+ T} \! \int_{\complement \BB(0,n)} \! \Bigl( \int_{\R^d} \! \ee^{-C_2 \abs{z}^{m/(m-1)}}\d z\Bigr)^{p/p'} \!\! \int_{\R^d} \ee^{-C_2 \abs{z}^{m/(m-1)}} \abs{f(x-t^{1/m}z)}^p \d z \d x \\
    & \leq C_1^p\ee^{p\omega_+ T} \Bigl( \int_{\R^d} \ee^{-C_2 \abs{z}^{m/(m-1)}}\d z\Bigr)^{p/p'} \!\! \int_{\R^d} \int_{\complement \BB(0,n)}  \!\!\!\!\! \ee^{-C_2 \abs{z}^{m/(m-1)}}  \abs{f(x-t^{1/m}z)}^p\d x \d z.
  \end{align*}
  Let us focus on estimating the double integral over $\R^d \times \complement \BB(0,n)$ in the previous calculation by splitting it up. 
  To this end, let $\varepsilon>0$. 
  Then there exist $n_0\in\N$ and $R>0$ such that 
  \begin{equation*}\int_{\complement \BB(0,n_0)} \abs{f(y)}^p\d y \leq \varepsilon
    \quad\text{and}\quad 
    \int_{\complement \BB(0,R)} \ee^{-C_2 \abs{z}^{m/(m-1)}}\d z \leq \varepsilon \;.
  \end{equation*}
  Consequently, for $n\geq n_0 + T^{1/m} R$, we have $\complement \BB(0,n) - t^{1/m} \BB(0,R) \subseteq \complement \BB(0,n_0)$ and
  \begin{align*}
    \int_{\R^d} \int_{\complement \BB(0,n)} &\ee^{-C_2 \abs{z}^{m/(m-1)}}  \abs{f(x-t^{1/m}z)}^p\d x \d z \\
    & = \int_{\BB(0,R)} \int_{\complement \BB(0,n)} \ee^{-C_2 \abs{z}^{m/(m-1)}}  \abs{f(x-t^{1/m}z)}^p\d x \d z \\
    &\qquad + \int_{\complement \BB(0,R)} \int_{\complement \BB(0,n)} \ee^{-C_2 \abs{z}^{m/(m-1)}}  \abs{f(x-t^{1/m}z)}^p\d x \d z \\
    & \leq \varepsilon \int_{\BB(0,R)} \ee^{-C_2 \abs{z}^{m/(m-1)}}\d z + \varepsilon \norm{f}_{\Ell^p(\R^d)}^p.
  \end{align*}
  Thus,
  \begin{align}
    \sup_{t\in [0,T]} \| \one_{(\thickset(t) - x_n)} (1 - \one_{\BB(0,n)}) p_{t,0}\ast f \|_{\Ell^p(\R^d)}^p \to 0 
    \label{eq:SecondTermToZero}
    \end{align}
  as $n$ tends to $\infty$. By \eqref{eq:notMeanThick}--\eqref{eq:SecondTermToZero} we obtain
  \begin{align*}
    \int_0^T \| ( U_p(t,0) f_n)|_{\thickset(t)} \|_{\Ell^p(\thickset(t))}^r \d t & \to 0 
  \end{align*}
  as $n$ tends to $\infty$.
  Since $\norm{U_p(T,0)f_n}_{\Ell^p(\R^d)} = \norm{p_{T,0}\ast f}_{\Ell^p(\R^d)} > 0$ for all $n\in\N$, an observability estimate does not hold.
\end{proof}

\begin{rem}
  \begin{enumerate}[(a)]
    \item
    Combining Theorem~\ref{thm:obs} and Theorem~\ref{thm:obsImpliesMeanThick}, we observe that uniformly thick observability sets allow for a final-state observability estimate, while such an estimate only implies that the observation sets are mean thick. It is an interesting question whether it is possible to close this gap, either by finding a suitable condition on the observation sets which is equivalent to a final-state observability estimate, or by proving that an observability estimate holds for mean thick sets. Even in the setting of Hilbert spaces and for autonomous problems, i.e.\ $p=r=2$ and $A(\cdot)$ time-independent, an answer on this question is still open. However, for 
    a certain class of non-autonomous diffusive evolution equations governed by the Ornstein--Uhlenbeck operator it has recently been proven in \cite{AlphonsoM-22} that the corresponding equation is cost-uniform approximate null-controllable, if and only if the family $(\Omega (t))_{t \in [0,T]}$ is mean thick. Here, cost-uniform approximate null-controllable is meant in the sense of property (\ref{item:null-non-unif}) in the introduction, and is thus a weaker property than the observability estimate in the above theorems. In fact, if $r=p=2$, then cost-uniform approximate null-controllability is equivalent to a so-called weak observability estimate, cf.\ \cite[Corollary~7.2]{AlphonsoM-22}.
    \item
    If the family of sets $(\thickset(t))_{t\in[0,T]}$ does not depend on $t$, then uniform thickness is equivalent to mean thickness. 
    In this case, one can prove the statement of Theorem~\ref{thm:obsImpliesMeanThick} for $r=\infty$ as well.
  \end{enumerate}
\end{rem}

\section{Observability for Non-autonomous Ornstein--Uhlenbeck Equations}\label{sec:OU}

This section applies the results from Section~\ref{sec:obs} to evolution families associated with non-autonomous Ornstein--Uhlenbeck equations.

\subsection{Non-autonomous Ornstein--Uhlenbeck Operators}

Let $p, r \in [1, \infty]$ and $T>0$. 
We consider \emph{non-autonomous Ornstein--Uhlenbeck equations} of the form
\begin{align}\label{eq:OrnsteinUhlenbeck}
\begin{split}
\dot{u}(t) - P(t)u(t) &= 0, \quad t \in (0,T], \quad u(0) = u_0 \in \Ell^p(\R^d)
\end{split}
\end{align}
with the \emph{non-autonomous Ornstein--Uhlenbeck operator} $P(t)$ given by
\begin{equation}\label{eq:OU_P}
P(t) \coloneqq \frac{1}{2}\trace(A(t)A(t)^\trans\nabla_x^2) - \scalarproduct{B(t)x}{\nabla_x} - \frac{1}{2}\trace(B(t))\,,
\end{equation}
where $A,B \in \CC^\infty((0,T);\R^{d \times d})$.

In general, the first-order term of these operators has coefficients that are allowed to vary in space.
Furthermore, they are not elliptic, so the theory of Section~\ref{sec:elliptic} does not apply. 
However, we will employ the following \emph{generalized Kalman rank condition} considered in~\cite{BeauchardEP-S-20} as a substitute for ellipticity.
More precisely, for $k \in \Z_+$ and $t \in [0,T]$, define $\widetilde A_k(t)$ by induction via the identities
\begin{equation*}
	\widetilde A_0(t) \coloneqq A(T -t), \quad \widetilde A_{k+1}(t) \coloneqq \frac{\mathrm{d}}{\mathrm{d} t}\widetilde A_k(t) + B(T -t)\widetilde A_k(t)\,.
\end{equation*}
We say that \emph{the generalized Kalman rank condition holds at time $T$} if
\begin{equation*}
	\spann~ \Big\{\widetilde A_k(T)x:~ x \in \R^d,k \in \Z_+\Big\} = \R^d\,.
\end{equation*}
It was shown in \cite[Section 6]{BeauchardEP-S-20} that, if the generalized Kalman rank condition holds at time $T$, then the problem \eqref{eq:OrnsteinUhlenbeck} admits a unique \emph{weak solution} $u \in \CC(0,T;\Ell^2(\R^d))$ and gives rise to an evolution family $(U_2(t,s))_{0 \leq s \leq t \leq T}$ on $\Ell^2(\R^d)$ defined by
\begin{equation}\label{eq:OU_Fourier}
	U_2(t,s)f \coloneqq \cF^{-1} \Big(\ee^{\frac{1}{2} \int\limits_s^t\trace{B(\tau)}\,\d \tau - \frac{1}{2}\int\limits_s^t \abs{A(\tau)^\trans R(t,\tau)^\trans\cdot}^2\,\d \tau}(\mathcal{F}f)(R(t,s)^\trans\cdot) \Big) ,
\end{equation}
for all $f \in \Ell^2(\R^d)$,
where, $(R(t,s))_{0\leq s\leq t\leq T}$ is the unique family of $d\times d$-matrices such that, for all $s,t \in [0,T]$, we have
\begin{equation*}
	\partial_t R(t,s) = B(t)R(t,s), \quad R(s,s) = \id_{\R^d}\,.
\end{equation*}
We recall that, for all $r,s,t \in [0,T]$, it holds that
\begin{equation}\label{eq:matrix group}
	R(t,s)R(s,r) = R(t,r)\,,
\end{equation}
see, e.g.,~\cite[Proposition~15]{Coron}.
For details on the above concept of weak solution, we refer the reader to~\cite[Section~6]{BeauchardEP-S-20} and \cite[Appendix~A.1]{BeauchardP-S}.

In the following, we will show that the evolution family $(U_2(t,s))_{0 \leq t \leq T}$, when restricted to $\Ell^p(\R^d) \cap \Ell^2(\R^d)$, can be extended to an evolution family on $\Ell^p(\R^d)$.
As a first step, we rewrite the exponent in~\eqref{eq:OU_Fourier} through a suitable quadratic form $q_{t,s}$.
Let $0 \leq s \leq t \leq T$ and recall that, in \cite[Proposition~15, Equation~(80)]{BeauchardEP-S-20}, the authors proved that there exist constants $c, \widetilde \varepsilon > 0$ and $m_1 \in \N$ such that, for all $\xi \in \R^d$ and $0 \leq s \leq t \leq \widetilde \varepsilon$, it holds that
\begin{equation}\label{eq:beps1}
	\int_s^t \abs{A(\tau)^\trans R(t,\tau)^\trans\xi}^2\d\tau \geq c\, (t - s)^{m_1}\abs{\xi}^2.
\end{equation}
Defining the matrix
\begin{equation*}
	Q_{t,s} \coloneqq \int_s^t R(s,\tau)A(\tau)A(\tau)^\trans R(s,\tau)^\trans\d \tau,
\end{equation*}
we have
\begin{align*}
	\int_s^t \abs{A(\tau)^\trans R(t,\tau)^\trans\xi}^2\d \tau 
  &= \int_s^t \scalarproduct{R(t,\tau)A(\tau)A(\tau)^\trans R(t,\tau)^\trans\xi}{\xi}\d \tau \\ 
	&=\scalarproduct{Q_{t,s}R(t,s)^\trans\xi}{R(t,s)^\trans \xi}
  \eqqcolon q_{t,s}(\xi),
\end{align*}
where the latest identity is due to the transposed version of \eqref{eq:matrix group}
\begin{equation*}
	R(s,r)^\trans R(t,s)^\trans = R(t,r)^\trans .
\end{equation*}
It follows from \eqref{eq:beps1} that $q_{t,s}$ is a positive definite quadratic form for $0 \leq  s \leq t \leq \widetilde \varepsilon$, and we may rewrite~\eqref{eq:OU_Fourier} via
\begin{equation}\label{eq:OU_Fourier_form}
	U_2(t,s)f \coloneqq \cF^{-1} \Big(\ee^{\frac{1}{2} \int\limits_s^t\trace{B(\tau)}\,\d \tau} \ee^{- \frac{q_{t,s}}{2}}(\mathcal{F}f)(R(t,s)^\trans\cdot) \Big), \quad f \in \Ell^2(\R^d).
\end{equation}

To prove the desired $\Ell^p$-estimate for the operator $U_2(t,s)$, we employ the following lemma.

\begin{lem}\label{lem:Aqlambda}
	Let $p \in [1,\infty]$ and $\Lambda, Q \in \R^{d \times d}$ such that $\det(\Lambda) \neq 0$ and the quadratic form
	\begin{equation*}
		q\from \R^d \to \R,\quad \xi \mapsto \scalarproduct{Q\Lambda^\trans\xi}{\Lambda^\trans\xi}
	\end{equation*}
	is positive definite. Consider the operator $A_{q,\Lambda}\from \mathcal{S}(\R^d)\to \mathcal{S}(\R^d)$
	given by
	\begin{equation*}
		A_{q, \Lambda} f \coloneqq \mathcal{F}^{-1} \Big(\ee^{-q/2}(\mathcal{F}f)(\Lambda^\trans\cdot) \Big)\,.
	\end{equation*}
	Then, for all $p \in [1,\infty]$ and all $f \in \mathcal{S}(\R^d)$, it holds that
	\begin{equation*}
		\norm{A_{q,\Lambda} f}_{\Ell^p(\R^d)} \leq \abs{\det(\Lambda)}^{-1/p'}\norm{f}_{\Ell^p(\R^d)}\,,
	\end{equation*}
	where $1/p + 1/p' = 1$.
\end{lem}
\begin{proof}
	Clearly,
	\begin{equation*}
		A_{q,\Lambda} = \mathcal{F}^{-1}\ee^{-q/2}\mathcal{F}\mathcal{F}^{-1}(\mathcal{F}f)(\Lambda^\trans\cdot) =  \mathcal{F}^{-1}\ee^{-q/2}\mathcal{F}\frac{1}{\abs{\det(\Lambda)}}f( \Lambda^{-1}\cdot)\,.
	\end{equation*}
	Since, by substitution for $p<\infty$ and directly for $p=\infty$,
	\begin{equation*}
		\left \| \frac{1}{\det(\Lambda)}f( \Lambda^{-1}\cdot) \right \|_{\Ell^p(\R^d)} =\abs{\det(\Lambda)}^{1/p - 1} \norm{f}_{\Ell^p(\R^d)},
	\end{equation*}
	we only need to prove that
	\begin{equation*}
		\norm{\mathcal{F}^{-1}\ee^{-q/2}\mathcal{F}f}_{\Ell^p(\R^d)} \leq \norm{f}_{\Ell^p(\R^d)}\,.
	\end{equation*}
	By Young's inequality, it suffices to show that
	\begin{equation*}
		\norm{\mathcal{F}^{-1}\ee^{-q/2}}_{\Ell^1(\R^d)} \leq 1\,.
	\end{equation*}
	 Let $M$ be the symmetric positive definite matrix such that $q = \ska{M\cdot,\cdot}$ and set $\overline q \coloneqq \ska{M^{-1}\cdot,\cdot}$. It is well-known that 
	\begin{equation*}
		\mathcal{F}^{-1}\ee^{-q/2} = \frac{\ee^{-\overline q/2}}{(2\pi)^{d/2}\det(M)^{1/2}}\,;
	\end{equation*}
 see, for example, \cite[Theorem~7.6.1]{Hoermander90}.
	Again, by substitution, it follows that
	\begin{equation*}
		\norm{\mathcal{F}^{-1}\ee^{-q/2}}_{\Ell^1(\R^d)} = \frac{1}{(2\pi)^{d/2}}\int_{\R^d} \ee^{-\abs{y}^2/2}\d y = 1\, ,
	\end{equation*}
	which yields the assertion.
\end{proof}

Note that, from \emph{Liouville's formula}
\begin{equation}\label{eq:liouville}
	\det(R(t,s)) = \ee^{\int_s^t\trace{B(\tau)}\d \tau},
\end{equation}
cf.~\cite[Lemma~3.11]{Teschl}, it follows that $R(t,s)$ is invertible and, thus, we may rewrite identity~\eqref{eq:OU_Fourier_form} using the operator $A_{q_{t,s}, R(t,s)}$ from Lemma~\ref{lem:Aqlambda} as 
\begin{equation}\label{eq:OU_Fourier_Aqlambda}
	U_2(t,s)f = \ee^{\frac{1}{2} \int\limits_s^t\trace{B(\tau)}\d \tau}A_{q_{t,s}, R(t,s)}f, \quad f \in \mathcal{S}(\R^d).
\end{equation}
Furthermore, formula~\eqref{eq:liouville} and Lemma~\ref{lem:Aqlambda} give the estimate
\begin{equation*}
	\norm{U_2(t,s)f}_{\Ell^p(\R^d)} \leq \ee^{\left(\frac{1}{2} - \frac{1}{p'}\right)\int\limits_s^t\trace{B(\tau)}\d \tau} \norm{f}_{\Ell^p(\R^d)}\,.
\end{equation*}
Therefore, $U_2(t,s)$ extends to a bounded operator $U_p(t,s)\from \Ell^p(\R^d) \to \Ell^p(\R^d)$ with norm estimate
\begin{equation}\label{eq:OU_Norm}
		\norm{U_p(t,s)}_{\cL(\Ell^p(\R^d))} \leq \ee^{\left(\frac{1}{2} - \frac{1}{p'}\right)\int\limits_s^t\trace{B(\tau)}\d \tau}.
\end{equation}

\subsection{Observability}

Using a dissipation estimate from \cite{BeauchardEP-S-20}, we prove the following observability estimate for small final times $T$.
\begin{thm}\label{thm:OU-obs}
	Let $p \in (1,\infty)$ and $\widetilde{T} > 0$ such that the generalized Kalman rank condition holds at time~$\widetilde{T}$. Then there exists a constant $\widetilde \varepsilon \in (0,\widetilde{T}]$ such that, if $T\in [0,\widetilde{\varepsilon}\,]$, $(\thickset (t))_{t \in [0, T]}$ uniformly thick on $[0, T]$, $E \subseteq [0,T]$ measurable with positive Lebesgue measure, and $r\in [1,\infty]$, then there exists $C_{\mathrm{obs}}\geq 0$ such that, for all $u_0\in \Ell^p(\R^d)$, we have
	\[\norm{U_p(T,0)u_0}_{\Ell^p(\R^d)} \leq C_{\mathrm{obs}}\begin{cases}
                                            \Bigl(\int_{E} \norm{(U_p(t,0)u_0)|_{\thickset(t)}}_{\Ell^p(\thickset(t))}^r \d t\Bigr)^{1/r}, & r\in [1,\infty),\\
                                            \esssup_{t\in E} \norm{(U_p(t,0)u_0)|_{\thickset(t)}}_{\Ell^p(\thickset(t))}, & r=\infty,
	                                     \end{cases}\]
	where $(U_p(t,s))_{0 \leq s \leq t \leq T}$ is the evolution family defined in~\eqref{eq:OU_Fourier_Aqlambda}.
\end{thm}

Note that, in contrast to Theorem \ref{thm:obs}, the constant $C_{\mathrm{obs}}$ may (and will) depend on $p$ (and $\widetilde{\varepsilon}\, $).

\begin{proof}[Proof of Theorem~\ref{thm:OU-obs}]
	We check that Hypothesis \ref{hypo:evoFam} is satisfied for the choice of $X = Y = \Ell^p(\R^d)$, $C(t)$ the restriction operator to $\thickset(t)$, $(P_\lambda)_{\lambda > 0}$ the family of smooth frequency cutoffs as defined in the proof of Theorem \ref{thm:obs}, and $(U(t,s))_{0 \leq s \leq t \leq  \widetilde T} = (U_p(t,s))_{0 \leq s \leq t \leq \widetilde T}$ the evolution family on $\Ell^p(\R^d)$ associated with a non-autonomous Ornstein--Uhlenbeck equation as above.
	
	The uncertainty principle follows directly from the Logvinenko--Sereda theorem \cite[Theorem~3]{Kovrijkine-01}, so we merely need to check the dissipation estimate. 
  To this end, define the \emph{sharp} spectral cutoff operator
	\begin{equation*}
		Q_\lambda \colon \Ell^2(\R^d) \to \Ell^2(\R^d), \quad f \mapsto \mathcal{F}^{-1}\mathbf{1}_{[-\lambda , \lambda]^d}\mathcal{F}f\,.
	\end{equation*} From \cite[Proposition~15]{BeauchardEP-S-20}, it follows immediately that there exist constants $c_0,c_1, \widetilde \varepsilon > 0$, and $m_1 \in \N$ such that, for all $0 \leq s \leq t \leq T \coloneqq \widetilde \varepsilon$, we have
	\begin{equation*}
		\norm{(\id - P_\lambda)U_2(t,s)}_{\cL(\Ell^2(\R^d))} \leq \norm{(\id - Q_{\lambda/(2\sqrt{d})})U_2(t,s)}_{\cL(\Ell^2(\R^d))} \leq c_0\ee^{-c_1(t-s)^{m_1}\lambda^2},
	\end{equation*}
	where $m_1$ is as in \eqref{eq:beps1}. On the other hand, it follows from the norm estimate~\eqref{eq:OU_Norm} that, for all $0 \leq s \leq t \leq T$, we have
	\begin{align}\label{eq:l1}
    \begin{split}
		&\norm{(\id - P_\lambda)U_1(t,s)}_{\cL(\Ell^1(\R^d))} \\ 
    &\leq \norm{\id - P_\lambda}_{\cL(\Ell^1(\R^d))}\norm{U_1(t,s)}_{\cL(\Ell^1(\R^d))}
    \leq K \ee^{\frac{1}{2}\int\limits_s^t\trace{B(\tau)}\d \tau},
    \end{split}
	\end{align}
	where we used uniform boundedness of the family $(P_\lambda)_{\lambda > 0}$ to define the constant
	\[
	 K \coloneqq 1+\norm{\mathcal{F}^{-1} \chi_1}_{\Ell^{1}(\R^d)}.
	\]
	By the \emph{Riesz--Thorin interpolation theorem}, for $p \in (1,2)$, we obtain
	\begin{equation*}
			\norm{(\id - P_\lambda)U_p(t,s)}_{\cL(\Ell^p(\R^d))} \leq  (c_0\ee^{-c_1(t-s)^{m_1}\lambda^2})^{1 -\theta} K^\theta \ee^{\frac{\theta}{2}\int\limits_s^t\trace{B(\tau)} \d \tau},
	\end{equation*}
	where
	\begin{equation*}
		\frac{1}{p} = \frac{\theta}{1} + \frac{1- \theta}{2}, \quad \text{i.e.\ }\quad \theta = \frac{1}{p} - \frac{1}{p'} = \frac{2}{p}-1 \quad \text{and} \quad 1 - \theta = \frac{2}{p'} = 2-\frac{2}{p}\,,
	\end{equation*}
	and thus, setting
	\begin{equation*}
		M_p \coloneqq \max\limits_{0 \leq s \leq t \leq T} K^{\frac{2}{p}-1} \ee^{(\frac{1}{p} - \frac{1}{2})\int\limits_s^t\trace{B(\tau)} \d \tau},
	\end{equation*}
	we obtain
	\begin{equation*}
		\norm{(\id - P_\lambda)U_p(t,s)}_{\cL(\Ell^p(\R^d))} \leq M_p (c_0\ee^{-c_1(t-s)^{m_1}\lambda^2})^{2-2/p}\,.
	\end{equation*}
	This proves the necessary dissipation estimate in the case that $p \in (1,2)$.
	For the case $p\in (2,\infty)$, fix $q \in \R$ such that $p < q$. If we replace \eqref{eq:l1} by
	\begin{equation*}
		\norm{(\id - P_\lambda)U_{q}(t,s)}_{\cL(\Ell^q(\R^d))} \leq K \ee^{(\frac{1}{2}-\frac{1}{q'})\int\limits_s^t\trace{B(\tau)}\d \tau}
	\end{equation*}
	in the above argument, where $1/q+1/q' = 1$, we end up with
	\begin{equation}\label{eq:dissipationOU}
		\norm{(\id - P_\lambda)U_p(t,s)}_{\cL(\Ell^p(\R^d))} \leq  (c_0\ee^{-c_1(t-s)^{m_1}\lambda^2})^{1-\sigma}K^{\sigma} \ee^{\sigma(\frac{1}{2} - \frac{1}{q'})\int\limits_s^t\trace{B(\tau)}\d \tau}\, ,
	\end{equation}
	where
	\begin{equation*}
		\frac{1}{p} = \frac{1-\sigma}{2} + \frac{\sigma}{q}, \quad\text{i.e.\ }\quad \sigma = \frac{p-2}{p}\cdot \frac{q}{q - 2}\,.
	\end{equation*}
	Letting $q \to \infty$ in~\eqref{eq:dissipationOU}, we obtain
	\begin{equation*}
		\norm{(\id - P_\lambda)U_p(t,s)}_{\cL(\Ell^p(\R^d))} \leq (c_0\ee^{-c_1(t-s)^{m_1}\lambda^2})^{2/p} K^{1-2/p} \ee^{(\frac{1}{p} - \frac{1}{2})\int\limits_s^t\trace{B(\tau)}\d \tau}
	\end{equation*}
	and therefore, by setting
	\begin{equation*}
		N_p \coloneqq \max\limits_{0 \leq s \leq t \leq T} K^{1-2/p} \ee^{(\frac{1}{p} - \frac{1}{2})\int\limits_s^t\trace{B(\tau)}\d \tau},
	\end{equation*}
	we obtain the estimate
	\begin{equation*}
		\norm{(\id - P_\lambda)U_p(t,s)}_{\cL(\Ell^p(\R^d))} \leq N_p \,  (c_0\ee^{-c_1(t-s)^{m_1}\lambda^2})^{2/p} 
	\end{equation*}
	for $p \in (2,\infty)$.
	In either case, this proves the dissipation estimate. The claim now follows from Theorem~\ref{thm:observability}.
\end{proof}

We demonstrate the above result by an example concerning the (autonomous) Kolmogorov equation with time-dependent observation sets.

\begin{example}
\label{ex:Kolmogorov}
We consider the evolution family associated with the classical \emph{Kolmogorov equation}
\begin{align*}
\partial_t u(t,x,v) - \Delta_v u(t,x,v) + v\cdot \nabla_x u(t,x,v) &= 0, \quad (t,x,v) \in [0,T] \times \R^d \times \R^d, \\
u(0,x,v) &=u_0(x,v), \quad u_0 \in \Ell^p(\R^d \times \R^d)\,,
\end{align*}
which, in the notation of \eqref{eq:OU_P}, corresponds to the choice of
\begin{equation*}
A(t) = \begin{pmatrix}
0 & 0 \\ 0 & \sqrt{2}\, \id_{\R^d}
\end{pmatrix} ,\quad B(t) = \begin{pmatrix}
0 & \id_{\R^d} \\ 0 & 0
\end{pmatrix},\quad t\in [0,T]\,.
\end{equation*}
It follows that, for $0 \leq s \leq t \leq T$ and the form
\begin{equation*}
q_{t,s}(\xi,\eta) = 2(t-s)\abs{\eta}^2 + 2(t-s)^2\eta\cdot\xi +\frac{2}{3}(t-s)^3\abs{\xi}^2, \quad (\xi, \eta) \in \R^d \times \R^d,
\end{equation*}
the associated evolution family $(U(t,s))_{0\leq s\leq t\leq T}$ is given for $u \in \mathcal{S}(\R^d \times \R^d)$ by
\begin{equation*}
(\mathcal{F}U(t,s)u) (\xi,\eta) = \ee^{-\frac{1}{2}q_{t,s}(\xi,\eta)}(\mathcal{F}u)(\xi,\eta + (t-s)\xi), \quad (\xi, \eta) \in \R^d \times \R^d\,,
\end{equation*}
which we can again extend to $\Ell^p(\R^d\times \R^d)$ by density for $p\in [1,\infty)$.
Note that, for arbitrary choices of $T$, we have for all $0 \leq s \leq t \leq T$
\begin{equation*}
q_{t,s}(\xi,\eta) 
\geq 
\frac{4 - \sqrt{13}}{3} \, \min \big\{ T^{-2}, 1 \big\} \, (t - s)^3 \big(\abs{\xi}^2 + \abs{\eta}^2 \big)\,.
\end{equation*}
Therefore, estimate~\eqref{eq:beps1} holds for $\widetilde \varepsilon \coloneqq T$ in this case.
Now, let $(\Omega(t))_{t\in [0,T]}$ be uniformly thick on $[0,T]$, $E\subseteq [0,T]$ measurable with positive Lebesgue measure, and $r\in [1,\infty]$. Then Theorem~\ref{thm:OU-obs} yields for $p\in (1,\infty)$ a final-state observability estimate.
\end{example}

\begin{rem}
    Theorem~\ref{thm:OU-obs} and Example~\ref{ex:Kolmogorov} can be regarded as extensions of~\cite[Corollary~8(i)]{BeauchardEP-S-20} and~\cite[Proposition~9]{BeauchardEP-S-20}, as we can treat non-autonomous Ornstein--Uhlenbeck equations in $\Ell^p(\R^d)$ for $p\in (1,\infty)$.
\end{rem}

\appendix
\section{Properties of Non-autonomous Elliptic Operators}
\label{sec:properties}

Let $T>0$, and let $\fa$ be a uniformly strongly elliptic polynomial of degree $m\geq 2$ with coefficients $a_\alpha\in \Ell^\infty(0,T)$ for $\abs{\alpha}\leq m$.
In this appendix, we collect some properties of $(U_p(t,s))_{0\leq s\leq t \leq T}$, where $(U(t,s))_{0\leq s\leq t \leq T}$ is as in \eqref{eq:defnEvoFam} and $p\in [1,\infty]$. 

\subsection{Kernel Estimates and Exponential Boundedness}

We will show that we can find \emph{Gaussian bounds} for the kernel $p_{t,s}$, $0\leq s <t \leq T$, from \eqref{eq:kernel}. We do this in two steps.

\begin{lem}
  \label{lem:real_part_bound}
    Let $c>0$ be the uniform ellipticity constant of $\fa$ from~\eqref{eq:ellipticPolynomial}. Then, for all $c_0\in (0,c)$, there exists $\omega\in\R$ such that
  \[\Re \fa(t,\xi) \geq c_0\abs{\xi}^m - \omega ,\quad \text{a.e. } t\in [0,T], \xi\in\R^d.\]
\end{lem}

\begin{proof}
	We have $\Re \fa_m(t,\xi) \geq c\abs{\xi}^m$ for all $t\in[0,T]$ and $\xi\in\R^d$. 
        Thus, for $c_0\in (0,c)$, we estimate
	\[
          \Re \fa(\cdot,\xi) \geq  c_0\abs{\xi}^m + (c-c_0)\abs{\xi}^m + \sum_{\abs{\alpha} < m} \Re (a_\alpha(\cdot) \ii^\alpha ) \xi^\alpha .
        \]
	Since the coefficient functions are locally essentially bounded, it is easy to see that there exists $\omega\in\R$ such that
	\[(c-c_0)\abs{\xi}^m + \sum_{\abs{\alpha} < m} (\Re a_\alpha(t)\ii^\alpha) \xi^\alpha \geq -\omega ,\quad \text{a.e. } t\in [0,T], \xi\in\R^d.\]
	This yields the assertion.
\end{proof}

\begin{lem}
	\label{lem:kernel_bound}
        There exist $C_1,C_2\geq0$, and $\omega\in\R$ such that, for all $0 \leq s < t \leq T$ and all $x \in \R^d$, we have
	\[
	\abs{p_{t,s}(x)} \leq C_1 \frac{1}{(t-s)^{d/m}} \ee^{\omega(t-s)} \ee^{-C_2 \bigl(\abs{x}^m/(t-s)\bigr)^{1/m-1}} .
	\]
        In particular, for all $p\in [1,\infty]$ and $0 \leq s < t \leq T$, we have
	\begin{align*}
		\norm{U_p(t,s)}_{\cL(\Ell^p(\R^d))} = \norm{p_{t,s}}_{\Ell^1(\R^d)} \leq C_1 \ee^{\omega (t-s)} \int_{\R^d} \ee^{-C_2 \abs{x}^{m/(m-1)}}\,\d x.
	\end{align*}
\end{lem}

\begin{proof}
	We follow the argument in \cite[Proposition~2.1]{TerElstR-96}.
        Note that, although $\fa(t,\cdot)$ is defined on $\R^d$, since it is a polynomial, we can extend it to $\C^d$ for all $t\in[0,T]$.
	Let $0\leq s< t\leq T$, $x\in\R^d$.
        Then, for $\eta\in\R^d$, we obtain via the change of variables formula
	\[p_{t,s}(x) = \frac{1}{(2\pi)^{d}} \int_{\R^d} \ee^{\ii x\cdot\xi} \ee^{-x\cdot\eta}\ee^{-\int\limits_s^t \fa(\tau,\xi+\ii\eta)\d \tau}\d\xi.\]
        In view of Lemma~\ref{lem:real_part_bound}, there exist $c_0,c_1,c_2>0$ such that, for almost all $\tau \in [0,T]$ and all $\xi , \eta \in \R^d$, we have
	\begin{align*}
        \Re \fa(\tau,\xi+\ii\eta) & = \Re \fa(\tau,\xi) + \Re \fa(\tau,\ii \eta) + \Re a_0(\tau) \\
        &\quad + \Re \sum_{\genfrac{}{}{0pt}{2}{\abs{\alpha}\leq m}{\alpha\neq 0}} a_\alpha(\tau) \sum_{\genfrac{}{}{0pt}{2}{\beta\leq \alpha}{\beta \neq 0,\alpha}} \genfrac{(}{)}{0pt}{0}{\alpha}{\beta} (\ii \xi)^\beta (-\eta)^{\alpha-\beta}\\
        & \geq c_0 \abs{\xi}^m - c_1 \abs{\eta}^m - c_2\Bigl(1+\sum_{1\leq k\leq m} \sum_{\genfrac{}{}{0pt}{2}{1\leq l\leq k}{l\neq k}} \abs{\xi}^l \abs{\eta}^{k-l}\Bigr).
    \end{align*}
    Note that, in order to estimate $ \Re \fa(\tau,\ii \eta)$, the fact that $m$ is an even number is crucial as it implies $\fa_m(\tau, \ii \eta) = (-1)^m \fa_m(\tau, \eta)$ which follows directly from Definition~\ref{defn:ellipticPolynomial}.
    Now, by Young's inequality for products, we can choose $\omega\geq 0$ such that
    \[c_2\Bigl(1+\sum_{1\leq k\leq m} \sum_{\genfrac{}{}{0pt}{2}{1\leq l\leq k}{l\neq k}} \abs{\xi}^l \abs{\eta}^{k-l}\Bigr)\leq \frac{c_0}{2} \abs{\xi}^m + \omega(1+\abs{\eta}^m).\]
    Thus, we finally arrive at
    \[ \Re \fa(\tau,\xi+\ii\eta) \geq \frac{c_0}{2} \abs{\xi}^m - \sigma\abs{\eta}^m - \omega,\]
    where $\sigma \coloneqq c_1 + \omega$.
	Hence, we can estimate
	\begin{align*}
		\abs{p_{t,s}(x)} & \leq \frac{1}{(2\pi)^{d}} \int_{\R^d} \ee^{-x\cdot\eta}\ee^{-\int_s^t \Re \fa(\tau,\xi+\ii\eta)\d \tau} \d\xi \\
		& \leq \frac{1}{(2\pi)^{d}} \int_{\R^d} \ee^{-x\cdot\eta}\ee^{-(t-s)(\frac{c_0}{2}\abs{\xi}^m - \sigma\abs{\eta}^m - \omega)} \d\xi \\
		& = C_1 \frac{1}{(t-s)^{d/m}}\ee^{-x\cdot\eta}\ee^{\omega(t-s)}\ee^{(t-s)\sigma\abs{\eta}^m},
	\end{align*}
	where $C_1\coloneqq \frac{1}{(2\pi)^{d}} \int_{\R^d} \ee^{-\frac{c_0}{2}\abs{\xi}^m} \d\xi$.
        Now, for $\eta\coloneqq \frac{1}{2}(\frac{\abs{x}}{\sigma(t-s)})^{1/(m-1)} \frac{x}{\abs{x}}$, we obtain
	\[\abs{p_{t,s}(x)} \leq C_1 \frac{1}{(t-s)^{d/m}} \ee^{\omega(t-s)}\ee^{-C_2 \bigl(\abs{x}^m/(t-s)\bigr)^{1/m-1}},
	\]
	where $C_2\coloneqq \frac{2^{m-1}-1}{2^m}$.
	Thus, integration yields the assertion for $\norm{p_{t,s}}_{\Ell^1(\R^d)}$. For $p\in[1,\infty]$, the operator $U_p(t,s)$ is the convolution operator with kernel $p_{t,s}$. Thus, we have $\norm{U_p(t,s)}_{\cL(\Ell^p(\R^d))} = \norm{p_{t,s}}_{\Ell^1(\R^d)}$.
\end{proof}

\subsection{Strong Continuity}

We now show that $(U_p(t,s))_{0\leq s\leq t \leq T}$ is strongly continuous for $p\in [1,\infty)$, while $(U_\infty(t,s))_{0\leq s\leq t \leq T}$ is strongly continuous with respect to the weak$^*$-topology.

We start with a subspace of $\cS'(\R^d)$ that can be identified with a subset of $\CC^\infty(\R^d)$.
Let $\cO_\mathrm{M}(\R^d)$ denote the \emph{multiplier space} 
\begin{align*}
  \cO_\mathrm{M}(\R^d) \coloneqq \big\{ f \in \CC^\infty(\R^d) : \forall g \in \cS(\R^d), \alpha \in \N_0^d:  \|f\|_{g,\alpha} < \infty\big\},
\end{align*}
where the family of seminorms $(\|\cdot\|_{g, \alpha})_{g,\alpha}$ is defined via
\begin{align*}
  \|f\|_{g, \alpha} \coloneqq \sup_{x \in \R^d} \big| g(x) \, \partial^\alpha f(x) \big| 
  ,\quad f\in \CC^\infty(\R^d),
\end{align*}
and induces a locally convex topology on $\cO_\mathrm{M}(\R^d)$, cf.\@ \cite[Chapter~7, §5, p.~243]{Schwartz-78}, \cite[Example~5.3]{Kruse-19}. Note that the multiplication $\cO_\mathrm{M}(\R^d)\times \cS(\R^d)\ni (f,g)\mapsto fg\in\cS(\R^d)$ is \emph{hypocontinuous}, in particular separately continuous, see, e.g,.~\cite{Larcher-13}.

\begin{prop}\label{prop:omconvergence}
    Let $(f_n)_{n\in\N}$ be in $\CC^\infty(\R^d)$ and $f \in \CC^\infty(\R^d)$ such that, for all $\alpha \in \N_0^d$, we have $\sup_{n \in \N} \|\partial^\alpha f_n\|_\infty < \infty$ and $\partial^\alpha f_n\to \partial^\alpha f$ uniformly on compact sets.
  Then $(f_n)_n$ is in $\cO_\mathrm{M}(\R^d)$, $f\in \cO_\mathrm{M}(\R^d)$, and $f_n \to f$ in $\cO_\mathrm{M}( \R^d)$.
\end{prop}

\begin{proof}
  Note that $\partial^\alpha f$ is bounded for all $\alpha\in\N_0^d$. 
  Since smooth functions whose derivatives of all orders are bounded clearly belong to $\cO_\mathrm{M}(\R^d)$, we obtain that $(f_n)_n$ is in $\cO_\mathrm{M}(\R^d)$ and $f\in \cO_\mathrm{M}(\R^d)$.
  
  Let $\alpha\in \N_0^d$ and $g \in \cS(\R^d)$, and let $\varepsilon > 0$. Choose a compact subset $K \subseteq \R^d$ such that
  \begin{align*}
    \sup_{x \not\in K} \abs{g(x)} \leq \frac{\varepsilon}{ 2 \sup_{n \in \N} \|\partial^\alpha (f_n - f)\|_\infty +1} .
  \end{align*}
    Furthermore, choose $N \in \N$ such that, for all $n \geq N$, we have
  \begin{align*}
    \sup_{x \in K} \big|\partial^\alpha (f_n - f)(x)\big|
    \leq 
    \frac{\varepsilon}{2 \, \| g \|_\infty+1} .
  \end{align*}
  Then we observe for all $n \geq N$
  \begin{align*}
    \norm{f_n-f}_{g,\alpha} & = \sup_{x\in\R^d} \big|g(x) \, \partial^\alpha (f_n - f)(x) \big|\\
    & \leq
      \sup_{x \in K} \big|g(x)\, \partial^\alpha(f_n - f)(x)\big|
      + \sup_{x \not\in K} \big|g(x) \, \partial^\alpha(f_n - f)(x)\big|\\
    & \leq \frac{\varepsilon}{2} + \frac{\varepsilon}{2} = \varepsilon  .
  \end{align*}
  Thus, $f_n \to f$ in $\cO_\mathrm{M}(\R^d)$.
\end{proof}

\begin{cor}\label{cor:stronglyContinuousFam}
 $(U(t,s))_{0\leq s\leq t \leq T}$ is strongly continuous on $\cS(\R^d)$.
\end{cor}

\begin{proof}
  Let $0\leq s\leq t\leq T$. 
  Let $((t_n,s_n))_{n \in \N}$ be in $[0,\infty)^2$, with $0\leq s_n\leq t_n\leq T$ for all $n\in\N$ and limit $(t_n,s_n)\to (t,s)$. 
  For $n\in\N$, set 
  \begin{align*}
    f_n \coloneqq \ee^{-\int\limits_{s_n}^{t_n} \fa(\tau,\cdot)\d \tau} \in \CC^\infty(\R^d) .
  \end{align*}
    Let furthermore $f\coloneqq \lim_{n \to \infty} f_n$ denote the pointwise limit. By the uniform strong ellipticity of $\fa$, it follows that the convergence of $(f_n)_n$ and its partial derivatives is also uniform on compact subsets of $\R^d$ (with the limit being the corresponding partial derivative of $f$) and that, for all $\alpha \in \N_0^d$, the sequence $(\partial^\alpha f_n)_{n}$ is bounded, see also Lemma \ref{lem:real_part_bound}. By Proposition~\ref{prop:omconvergence}, we have that $f_n\to f$ in $\cO_\mathrm{M}(\R^d)$.

  Let $g \in \cS(\R^d)$. Then $\cF(g)\in \cS(\R^d)$ and therefore $f_n\cF(g)\to f\cF(g)$ in $\cS(\R^d)$. Thus,
  \begin{align*}
    U(t_n,s_n) g = \cF^{-1} (f_n  \cF (g)) \to \cF^{-1} (f \cF (g)) =  U(t,s)g
  \end{align*}
  by continuity of $\cF^{-1}$.
\end{proof}

\begin{cor}\label{cor:strong_continuity}
    Let $p\in[1,\infty)$. Then $(U_p(t,s))_{0\leq s\leq t \leq T}$ is strongly continuous. Moreover, $(U_\infty(t,s))_{0\leq s\leq t \leq T}$ is strongly continuous with respect to the weak$^*$-topology.
\end{cor}

\begin{proof}
Note that Lemma \ref{lem:kernel_bound} yields uniform boundedness of $(U_p(t,s))_{0\leq s\leq t\leq T}$ for all $p\in[1,\infty]$.

Since $\cS(\R^d)\hookrightarrow\Ell^p(\R^d)$ is dense for $p\in [1,\infty)$, Corollary~\ref{cor:stronglyContinuousFam} yields that the family $(U_p(t,s))_{0\leq s\leq t\leq T}$ is strongly continuous for $p\in[1,\infty)$.
For $0\leq s\leq t\leq T$, we have $U_\infty(t,s) = V_1(t,s)'$, where $(V_1(t,s))_{0\leq s\leq t\leq T}$ is the evolution family on $\Ell^1(\R^d)$ associated with the non-autonomous polynomial 
  \begin{equation*}
    \fa(\cdot,-\cdot) \colon (t,\xi)\mapsto \fa(t,-\xi) = \sum_{\abs{\alpha}\leq m} (-1)^{\abs{\alpha}}a_\alpha (t) \ii^\alpha \xi^\alpha
  \end{equation*}
  which is also uniformly strongly elliptic since $m$ is necessarily even. Thus, the second assertion follows.
\end{proof}

\subsection{\texorpdfstring{$\boldsymbol{(U_p(t,s))_{0\leq s\leq t\leq T}}$}{(U\_p(t,s))} as an Evolution Family for \texorpdfstring{$\boldsymbol{(A_p(t))_{t\in[0,T]}}$}{(A\_p(t))}}

Now, we establish the relation between the evolution family $(U_p(t,s))_{0\leq s\leq t \leq T}$ and the family of differential operators $(A_p(t))_{t \in [0,T]}$.

\begin{prop}
\label{prop:evoFamBounded}
        Let $p\in (1,\infty)$ and $u\in D^p$.
        \begin{enumerate}[(a)]
            \item\label{it:leftAp} Let $0 \leq s < T$. Then $U_p(\cdot,s)u\in \WW^{1,1}(s,T; \Ell^p(\R^d))\cap \Ell^1(s,T; D^p)$ and, for almost all $t \in (s,T)$, we have $\partial_t(U_p(t,s)u) = - A_p(t) U_p(t,s) u$.
            \item\label{it:rightAp} Let $0 < t \leq T$. Then $U_p(t,\cdot)u\in \WW^{1,1}(t,T; \Ell^p(\R^d))\cap \Ell^1(t,T; D^p)$ and, for almost all $s \in (t,T)$, we have $\partial_s(U_p(T,s)u) = U_p(T,s) A_p(s) u$.
	\end{enumerate}
\end{prop}

\begin{proof}
  \eqref{it:leftAp} By Young's inequality and the $\Ell^1$-bound of the kernel in Lemma \ref{lem:kernel_bound}, we observe $U_p(\cdot,s)u\in \Ell^1(s,T;\Ell^p(\R^d))$.	
	Note that $(s,T) \ni t\mapsto p_{t,s}$ is weakly differentiable and $\partial_t p_{t,s} = -A_p(t)p_{t,s}$ for all $t\in (s,T)$.
	For the weak derivative of $U_p(\cdot,s)u$, we have   
	\begin{align*}
		\partial_t(U_p(t,s) u) & = \partial_t(p_{t,s} \ast u) = (\partial_t p_{t,s}) \ast u = (-A_p(t) p_{t,s}) \ast u\\
                & = p_{t,s} \ast(-A_p(t)u) = -A_p(t) (p_{t,s}\ast u) = -A_p(t) U_p(t,s)u
	\end{align*}
	for almost all $t\in (s,T)$.
        In particular, the closedness of $A_p(t)$ implies $U_p(t,s) u \in D^p$ for almost all $t \in (s,T)$.
	
        Note that $D^p = \WW^{p,m}(\R^d)$ and the Sobolev norm and the graph norms $\|\cdot \|_{A_p(t)}$ are equivalent, i.e.\@ there exists $C>0$ such that
	\[\frac{1}{C}\norm{v}_{\WW^{p,m}(\R^d)}\leq \norm{v}_{A_p(t)} \leq C\norm{v}_{\WW^{p,m}(\R^d)}\]
	for all $t\in [s,T]$ and $v\in D^p$ which follows from uniform strong ellipticity of $\fa$ and the boundedness of the coefficients $a_\alpha$.
	In particular, $\norm{A_p(t)u}_{\Ell^p(\R^d)} \leq C \norm{u}_{\WW^{p,m}(\R^d)}$ for all $t\in[s,T]$ and $u \in D^p$, and therefore
	\[\int_s^T \norm{A_p(t) u}_{\Ell^p(\R^d)} \d t < \infty,\]
        so, by Young's inequality and the kernel bound from Lemma \ref{lem:kernel_bound}, we observe that $\partial_t U_p(\cdot,s)u\in \Ell^1(s,T;\Ell^p(\R^d))$ as well as $U_p(\cdot,s)u\in \Ell^1(s,T;D^p)$.
	
        The proof of~\eqref{it:rightAp} follows the same lines as the proof of~\eqref{it:leftAp}.
\end{proof}

\phantomsection
\section*{Acknowledgement}
F.G.\ thanks K.~Kruse for valuable discussions about the multiplier space $\cO_\mathrm{M}(\R^d)$ and D.~Gallaun for a helpful conversation on interpolation theory.
Furthermore, the authors would like to thank the anonymous reviewers for their comments and suggestions.

\end{document}